\documentclass[11pt]{amsart}
\usepackage{hyperref}
\hypersetup{
    colorlinks=true,
    citecolor=blue,%
    linkcolor=blue,
    filecolor=blue,      
    urlcolor=blue,
}

 \usepackage[margin=1in,footskip=0.25in]{geometry}

\usepackage{amsmath}
\usepackage{amssymb}
\usepackage{graphicx}

\usepackage{caption}
\usepackage{subcaption}
\usepackage{setspace}
\usepackage{array}
\newcolumntype{P}[1]{>{\centering\arraybackslash}p{#1}}
\usepackage{amsmath}
\usepackage{amssymb,amsfonts,amstext}
\usepackage{commath}
\usepackage{graphics}
\usepackage{csquotes}
\usepackage{hyphenat}
\usepackage{mathrsfs}
\usepackage{algorithm}
\usepackage{multirow}
\usepackage{pdflscape}
\usepackage{algpseudocode}
\usepackage{enumitem}
\usepackage{tcolorbox}
\usepackage{rotating}
\usepackage{algorithmicx}
\usepackage{algpseudocode}
\usepackage{float}
\usepackage{mathtools}
\usepackage{cite}

\newtheorem{theorem}{Theorem}[section]
\newtheorem{corollary}[theorem]{Corollary}
\newtheorem{lemma}[theorem]{Lemma}
\newtheorem{proposition}[theorem]{Proposition}
\theoremstyle{definition}
\newtheorem{definition}[theorem]{Definition}
\newtheorem{example}[theorem]{Example}
\newtheorem{remark}[theorem]{Remark}
\newtheorem{note}{Note}

\newcommand{\sgn}{\mathop{\mathrm{sgn}}}

\title[GOWA robustness for uncertain optimization problems]{Generalized Ordered Weighted Aggregation Robustness to Solve Uncertain Single Objective Optimization Problems}

\author[N. Kishor]{Nand Kishor}
\address[N. Kishor]{Department of Mathematical Sciences, Indian Institute of Technology (BHU), Varanasi 221005, India}
\email{\tt nandkishor.rs.mat20@itbhu.ac.in}

\author[D. Ghosh]{Debdas Ghosh}
\address[D. Ghosh]{Department of Mathematical Sciences, Indian Institute of Technology (BHU), Varanasi 221005, India }
\email{{\tt debdas.mat@iitbhu.ac.in}}

\author[X. Zhao]{Xiaopeng Zhao$^*$}
\address[X. Zhao]{Corresponding Author, School of Mathematical Sciences, Tiangong University, Tianjin, 300387, China}
\email{{\tt zhaoxiaopeng.2007@163.com}}



\begin{document}

\begin{abstract}

Robust optimization aims to find optimum points from the collection of points that are feasible for every possible scenario of a given uncertain set. An optimum solution to a robust optimization problem is commonly found by the min-max robust counterpart or by the \emph{best out of the worst-cases} analysis. In this article, we introduce a new counterpart with the help of the \emph{generalized ordered weighted aggregation} (GOWA) operator to solve uncertain single objective optimization problems. After introducing GOWA robustness, we analyze a few elementary properties of the GOWA robust objective function, like continuity, monotonicity, coerciveness, local Lipschitz property, and subdifferential regularity. An approach to computing the Clarke subdifferential of the GOWA robust objective function is also provided. We discuss the relationship between the concept of GOWA robustness with other existing robustness---flimsily, highly, min-max, light, and min-min robustness. We show that in a particular case, GOWA robustness reduces to the commonly used min-max robustness. The entire paper is supported by several geometrical and numerical illustrations.

\end{abstract}

\maketitle
\section{Introduction}\label{Intro}

The robust optimization approach handles optimization problems where the involved data or parameters are uncertain but belong to a known uncertainty set. It is common to have uncertain data in realistic optimization problems  \cite{Robustapplication,Yu,Tiedemann,Eltoukhy}. Robust optimization \cite{Kouvelis} and stochastic optimization \cite{Birge} are two popular ways to handle data uncertainty in optimization problems. In stochastic optimization, it is essential to know the probability distribution (at least an estimate of it) of the uncertain data. On the other hand, robust optimization methods assume that the uncertain data lies in an uncertainty set rather than the actual probability distribution is known. Robust optimization is popular because several classes of uncertain optimization problems can be made computationally tractable using the approaches of robust optimization \cite{Ben}.

In the recent past, robust optimization has been widely used in realistic optimization problems, e.g., portfolio optimization \cite{Fliege}, route planning \cite{Kuhn}, public transport \cite{Robustness}, irrigation network \cite{Irrigation}, etc. An NETLIB case study of robust optimization models can be found in \cite{Robustapplication}. Study results \cite{Robustapplication} indicate that small perturbations of data can severely affect the feasibility properties of real-world optimization problems and that the robust optimization methodology can be successfully used to overcome this problem. Commonly, robust optimization searches for optimum solutions that are feasible for all scenarios.

In this paper, we aim for analyzing single-objective uncertain optimization problems. There are several approaches to finding a robust counterpart for solving such a problem, and one of the most familiar ones is the min-max (or strict) robust counterpart \cite{ Ben, soyster}. The min-max robust counterpart was originally introduced by Soyster \cite{soyster} and extensively studied by Ben-Tal and Nemirovski \cite{Ben, Ben-Tal, Ben-Tal A, Ben-Tal B}. In the context of the min-max approach, a robust minimum solution is such a point that minimizes the objective function in the worst case of all scenarios, and that is a feasible point under all uncertain scenarios. In the recent past, many alternatives to the min-max approach have been introduced. Goerigk and Schöbel \cite{Goerigk M} reported a recent overview of robustness concepts for single-objective problems. Light robustness is a concept introduced by Fischetti and Monaci \cite{Fischetti} and generalized by Sch{\"o}bel \cite{Schobel} that can be used as an alternative to min-max or regret robustness. In this concept, only solutions that would be suitable for the most likely scenarios are considered. By light robustness, it is thus possible to protect the best-case scenarios against perturbations because of its optimistic perspective. In selecting this solution, reliability is taken into account. There are also alternatives such as adjustable robustness \cite{Benn}, soft robustness \cite{Bertsmimas}, and recovery robustness \cite{Liebchen,Erera,Goerigk}. Sch{\"o}bel \cite{Schobel} further introduced min-min robustness, which gives the minimum most objective value for some scenarios. However, this method has the disadvantage that the optimum solution may not be feasible for all scenarios.

Although the min-max robust counterpart is the most popular among all the existing robust counterparts, the \emph{best out of the worst} (pessimistic) approach makes the min-max robustness restrictive because the min-max robust optimum solution does not minimize the objective value as much as possible for each scenario. We aim to alleviate this drawback of the min-max robust counterpart and propose a new robust counterpart to solve single-objective uncertain optimization problems. To this end, instead of minimizing the maximum function of the objective functions (min-max approach), we find an ordering of the objective functions in the feasible set and then minimize an aggregated function of the ordered functions.

For a given ordered collection of criteria, the employment of a suitable aggregating operator is of considerable importance in decision sciences. To choose an aggregation function, the ordered relationship between the criteria is a primary factor. The situation in which we desire that all criteria be satisfied is one extreme (the min-max approach crudely falls into this category). Another extreme is that we desire to satisfy only one of the criteria (the flimsily approach crudely falls into this category). Yager in \cite{Yager} introduced a family of operators called ordered weighted averaging (OWA) operators that provide an aggregation between these two extremes. The order weights depend on the decision maker’s degree of optimism \cite{Yager_2}.  Based on the weights for the OWA operator, Yager \cite{Yager_1} defined the degree of optimism. In \cite{Yager}, Yager further extended the ordered weighted averaging (OWA) operator to provide a new class of operators called the \emph{generalized OWA} (GOWA) operators.

In this paper, we employ GOWA operator to aggregate an ordering of the objective functions of all the scenarios of a given single objective uncertain objective function and propose a unified robustness called GOWA robustness. We further relate the proposed GOWA robustness with min-max robustness, light robustness, flimsily robustness, and min-min robustness. Neatly, the contributions of the paper are as follows:

\begin{itemize}
    \item We give a new method, called GOWA robustness, to solve uncertain single objective optimization problems and interpret the method with mathematical and geometrical details.
    \item It is found that, in particular cases, the objective function of GOWA robustness is equivalent to min-max and min-min robust counterparts. 
    \item In GOWA robustness, the optimal objective value is observed to be always lying in between the min-min and min-max robust optimal values. 
    \item As GOWA robust objective function is commonly nonsmooth, we analyze the Lipschitz property and the Clarke subdifferentiability of the newly introduced GOWA robust objective function.
\end{itemize}

The rest of the article is presented in the following manner. In Section \ref{section:2}, a definition of uncertain single objective optimization problems is provided. In Section \ref{section:3}, we propose the concept of the  GOWA robustness; alongside, a few necessary results on GOWA robustness with geometrical explanation are also given. Next, in Section \ref{relation}, we relate GOWA robustness with some other existing robustness. Further, the Clarke subdifferential for the objective function of GOWA robustness is presented in Section \ref{subdiffrential}. Lastly, the conclusion and future scopes are given in Section \ref{conclusion}.

%%%%%%%%%%%%%%%%%%%%%%%%%%%%%%%%%%%%%%%%%%%%%%%%%%%%%%%%

\section{Preliminaries} \label{section:2}
%%%%%%%%%%%%%%%%%%%%%%%%%%%%%%%%%%%%%%%%%%%%%%%
An uncertain single objective optimization problem
is a family of parameterized problems $\mathcal{P}(\Omega)=\{\mathcal{P}(\zeta): \zeta \in \Omega\},$ where 
\begin{align*}
\mathcal{P}(\zeta)~~\left\{\begin{aligned}%[3]
&\min    && f(x,\zeta)\\
&\text{subject to} \quad && x \in \mathcal{X} \subseteq \mathbb{R}^n.
\end{aligned}\right.
\end{align*}
For the class of problems $\mathcal{P}(\Omega),$ throughout the paper, we assume that
\begin{enumerate}
    \item[(i)] uncertainties in the problem formulation are given by finitely many $\zeta$'s, which constitute the set $\Omega \subseteq \mathbb{R}^m.$ Precisely, let $\Omega = \{\zeta_1,\zeta_2,\ldots,\zeta_p\}.$
    \item[(ii)] The feasible set $\mathcal{X}$ is independent of uncertainties $\Omega$, and unless otherwise stated, it is a compact subset of $\mathbb{R}^n.$
    \item[(iii)] For each given $\zeta \in \Omega,~ f(x,\zeta) \geq 0$ for all $x \in \mathcal{X}.$
\end{enumerate}
An element in $\Omega$ is referred to as a scenario. Note that the scenarios in $\Omega$ influence the values of the non-negative objective function $f:\mathcal{X} \times \Omega \to \mathbb{R}.$

In the following, we provide several concepts of robustness that transform the uncertain single objective optimization problem $\mathcal{P}(\Omega)$ into a deterministic optimization problem called its robust counterpart.
%%%%%%%%%%%%%%%%%%%%%%%%

\begin{definition} \cite{Robustness} 
For a given uncertain single objective optimization problem $\mathcal{P}(\Omega),$ a solution $x\in \mathcal{X}$ is called flimsily robust optimal solution for $\mathcal{P}(\Omega)$ if it is optimal for at least one $\zeta \in \Omega.$
\end{definition}

\begin{definition} \cite{Robustness}
For a given uncertain single objective optimization problem $\mathcal{P}(\Omega),$ an element $\bar{x} \in \mathcal{X}$ is called highly robust optimal solution for $\mathcal{P}(\Omega)$ if it is optimal for all $\zeta \in \Omega.$
\end{definition}

\begin{definition} \cite{min-max}
For a given uncertain single objective optimization problem $\mathcal{P}(\Omega),$ an element $\bar{x} \in \mathcal{X}$ is called a min-max robust optimal solution to $\mathcal{P}(\Omega)$ if it is an optimal solution to 
 \begin{align}\label{min-max}
\left\{\begin{aligned}
&\min   &&\underset{\zeta\in \Omega}{\max}~ f(x,\zeta)\\
&\text{subject to} \quad && x \in \mathcal{X}.\\
\end{aligned}\right.
\end{align}
This means a min-max robust optimal solution minimizes the worst case of the objective function under all possible scenarios.  The problem \eqref{min-max} is called the min-max robust counterpart of the problem $\mathcal{P}(\Omega)$.  
\end{definition}

Let us recall the following three definitions to analyze the subdifferentiability of $\phi(x)$.

%%%%%%%%%%%%%%%%%%%%%%%%%%%%%%%%%%%%%%
\begin{definition} \cite{Clarke}
Let $f : \mathbb{R}^n \to \mathbb{R}$ be a locally Lipschitz continuous function at $\bar{x} \in \mathbb{R}^n$. The Clarke generalized directional derivative of $f$ at $\bar{x}$ in the direction $d \in \mathbb{R}^n$
is defined by
\begin{align*}
    f^\circ(\bar{x};d)=\limsup_{
   \substack{x\rightarrow \bar{x}\\ t ~\downarrow~ 0}}\frac{f(x+td)-f(x)}{t}.
\end{align*}
\end{definition}
 %%%%%%%%%%%%%%%%%%%%%%%%%%%%%%%%%

 \begin{definition} \cite{Clarke}
 Let $f: \mathbb{R}^n \to \mathbb{R}$ be a locally Lipschitz continuous function at a point $\bar{x} \in \mathbb{R}^n.$ Then, the Clarke subdifferential of $f$ at $\bar{x}$ is the set
 $$\partial f(\bar{x})= \{\eta \in \mathbb{R}^n : f^{\circ}(\bar{x};d) \geq \eta^{\top} d \text{ for all } d \in \mathbb{R}^n \}.$$
Each vector $\eta \in \partial f(\bar{x})$ is called a subgradient of $f$ at $\bar{x}.$ 
 \end{definition}
%%%%%%%%%%%%%%%%%%%%%%%%%%%%%%%%%%%%%%%

\begin{definition} \cite{Clarke}
A function $f : \mathbb{R}^n \to \mathbb{R}$ is said to be subdifferentially regular at $\bar{x} \in \mathbb{R}^n$ if it is locally Lipschitz continuous at $\bar{x}$, the (conventional) directional derivative $f' (\bar{x}; d)$ exists for all $d \in \mathbb{R}^n$, and $f'(\bar{x}; d)=f^\circ(\bar{x}; d).$ If $f$ is subdifferentially regular at each point of $\bar{x} \in \mathcal{X},$ then $f$ is said to be regular on $\mathcal{X}.$
\end{definition}

\section{Introduction to GOWA robustness}\label{section:3}
%%%%%%%%%%%%%%%%%%%%%%%%%%%%%%%%%%%%%%%%%%%%%%%
To define GOWA robust counterpart of  $\mathcal{P}(\Omega)$, we adhere to the given ordered indexing of the scenarios in the given finite uncertainty set $\Omega = \{\zeta_1, \zeta_2, \ldots, \zeta_p\}$, i.e., $\Omega$ is a collection of $p$ elements of $\mathbb{R}^m$ which are in ordered indexing. 

Consider the collection of functions $f(x, \zeta_1)$, $f(x, \zeta_2), \ldots, f(x, \zeta_p)$ on $\mathcal{X}$. In min-max robust counterpart, we seek a minimum point of the function $m(x) = \max\{f(x, \zeta_1), f(x, \zeta_2), \ldots, f(x, \zeta_p)\}$ on $\mathcal{X}$. As the min-max approach is found to be restrictive, instead of minimizing $m(x)$, we find an ordering of the functions $f(x, \zeta_1), f(x, \zeta_2), \ldots, f(x, \zeta_p)$, and then go for minimizing a weighted average function of the ordered functions. To this end, we wish to order the functions  $f(x, \zeta_1)$, $f(x, \zeta_2), \ldots, f(x, \zeta_p)$ on $\mathcal{X}$ in terms of their function values. Commonly, there will be intersections of the graph of these functions. So, we cannot expect to find an ordering of the functions (e.g., $f(x, \zeta_2) \ge f(x, \zeta_5) \ge f(x, \zeta_1) \ge f(x, \zeta_p)$ for all $x \in \mathcal{X}$) which is applicable throughout the set $\mathcal{X}$.  
Thus, with the help of the functions $f(x, \zeta_1)$, $f(x, \zeta_2), \ldots, f(x, \zeta_p)$ on $\mathcal{X}$, we introduce a collection of auxiliary functions $s_i : \mathcal{X} \to \mathbb{R},$ $i=1,2,\ldots,p,$ so that 
\[s_1(x) \geq s_2(x) \geq \cdots \geq s_p(x) \text{ for all } x \in \mathcal{X}.\]
We define the functions $s_1, s_2, \ldots, s_p$ on $\mathcal{X}$ as follows: 
\begin{align}
&s_1(x)= \max\{f(x,\zeta_1),f(x,\zeta_2),\ldots,f(x,\zeta_p)\}~~\text{and}\label{s_1}\\
&s_i(x)= \max\bigg\{\{f(x,\zeta_1),f(x,\zeta_2),\ldots,f(x,\zeta_p)\}\mathbin{\Big\backslash}\{s_1(x),s_2(x),\ldots,s_{i-1}(x)\}\bigg\},~i=2,3,\ldots,p.\label{s_i}
\end{align}
For a given $\bar x \in \mathcal{X}$, in case there are more than one indices $j_1, j_2, \ldots, j_i$ with $1 \le j_1 < j_2 < \cdots < j_i \le p$  for which 
\begin{align*}
    f(\bar x, \zeta_{j_1}) = f(\bar x, \zeta_{j_2}) = \cdots = & f(\bar x, \zeta_{j_i})\\
    = & \max\bigg\{\{f(x,\zeta_1),f(x,\zeta_2),\ldots,f(x,\zeta_p)\}\mathbin{\Big\backslash}\{s_1(x),s_2(x),\ldots,s_{i-1}(x)\}\bigg\}, 
\end{align*}
then we take $s_i(\bar x)$ to be 
\begin{equation}\label{lowest_index} 
s_i(\bar x) = f(\bar x, \zeta_{j_1}).
\end{equation}
Instead of ${j_1}$, we can use any of the other indices $j_2, j_3, \ldots, j_i$. However, to maintain uniformity and convenience of discussion, we adhere to the equation (\ref{lowest_index}) with the lowest index number $j_1$. As the use of more than one indices out of $j_1, j_2, \ldots, j_i$ may lead to issues on well-definedness of $s_{i+1}(\bar x)$ (see Example \ref{s_i_plus_one_is_not_well_defined}), we adhere to the use of `only one' index of $\zeta$ in defining $s_i(\bar x)$.
 
 From the definition of $s_i$'s, we can easily observe that $s_1(x) \geq s_2(x) \geq\cdots\geq s_p(x)$ for all $x \in \mathcal{X}.$

\begin{example}\label{s_i_plus_one_is_not_well_defined}
Let the uncertainty set be $\Omega=\{\zeta_1, \zeta_2\}$ and the feasible set be $\mathcal{X}=[0,2]$.
We define $f:\mathcal{X}\times\Omega \to \mathbb{R}$ by
$f(x,\zeta_1)= (x-2)^2,$ and$ ~f(x,\zeta_2)= e^{x-1}.$ Graphs of $f(x,\zeta_1),~f(x,\zeta_2)$ and $s_1(x)$ are drawn in Figures \ref{s_2_is_not_defined_1} and  \ref{s_2_is_not_defined_2}, respectively. The auxiliary functions $s_1(x)$ and $s_2(x)$ are given by 
\begin{align*}
    s_1(x)
    & = \begin{dcases}
        (x-2)^2, & ~\text{if}~~ x \in [0,1] \\
        e^{x-1}, & ~\text{if}~~ x \in  [1,2],~~~\text{and}~~~
    \end{dcases}\\ 
    s_2(x)& = \begin{dcases}
        e^{x-1}, & ~\text{if}~~ x \in [0,1] \\
        (x-2)^2, & ~\text{if}~~ x \in  [1,2].
    \end{dcases}
\end{align*}
Notice that at $\bar{x}=1$, $f(\bar{x},\zeta_1)=f(\bar{x},\zeta_2).$ For the value of $s_1(1)$, if we use both of the indices $1$ and $2$, i.e., if we consider $s_1(1)$ to be both of $f(1, \zeta_1)$ and $f(1, \zeta_2)$, then \[ \bigg\{\{f(1, \zeta_1), f(1,\zeta_2)\}\mathbin{\Big \backslash}\{s_1(1)\}\bigg\}=\emptyset,\] and hence $s_2(x)$ will not defined at $\bar{x}=1$ (see Figure \ref{s_2_is_not_defined_3}). 
So, we adhere to apply the equation \eqref{lowest_index}, and take 
$s_1(1) = f(1, \zeta_1) = 1$. Then, $s_2(1)$ is well-defined and 
$s_2(1) = \max \big\{\{f(1, \zeta_1), f(1, \zeta_2)\} \setminus \{s_1(1)\}\big\} = f(1, \zeta_2) = 1$.

\noindent
\end{example}

\begin{figure}[H]
    \centering
    \subfloat[The objective functions for different scenarios\label{s_2_is_not_defined_1}]{\includegraphics[scale=0.4]{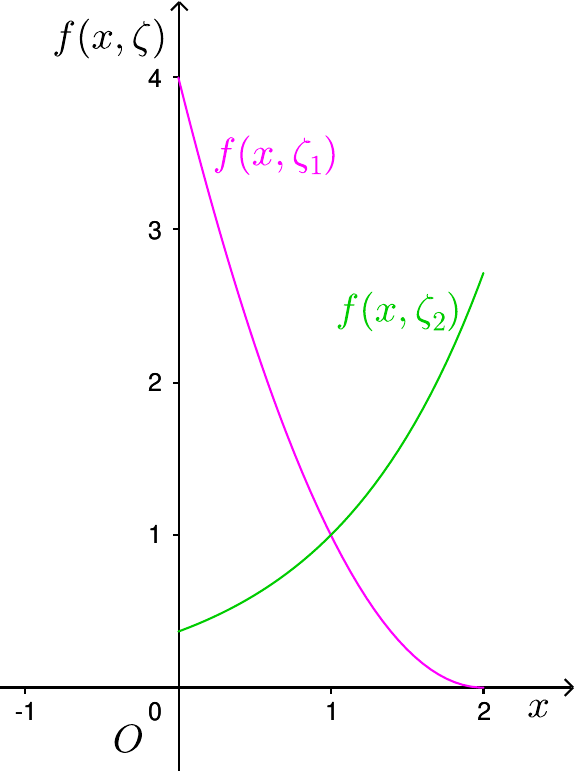}}
    \qquad \qquad 
    \subfloat[The function $s_1$\label{s_2_is_not_defined_2}]{\includegraphics[scale=0.4]{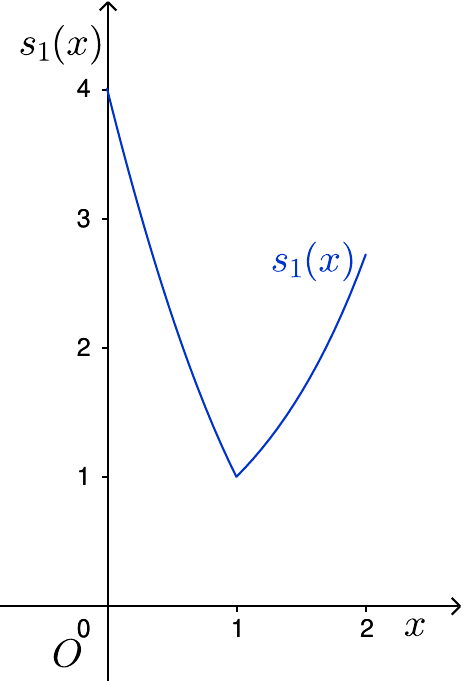}}
    \qquad \qquad 
    \subfloat[The function $s_2$\label{s_2_is_not_defined_3}]{\includegraphics[scale=0.4]{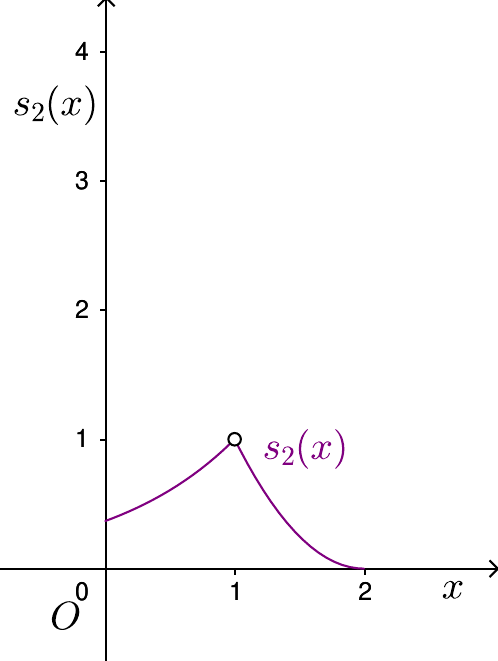}}
    \caption{The objective function and $s_1, s_2$ in Example \ref{s_i_plus_one_is_not_well_defined}}
\end{figure}
%%%%%%%%%%%%%%%%%%%%%%%%%%%%%%%%%%%%%%%%%%%%%%%%%%%%%%%%%%%%%%%%%%%%%%%%%%%%%%%

\begin{definition} \label{GOWAdefn}
(GOWA robust counterpart).  
An element $\bar{x} \in \mathcal{X}$ is called a GOWA robust optimal solution to $\mathcal{P}(\Omega)$ if it is an optimal solution to
 \begin{align}\label{GOWA}
\left\{\begin{aligned}%[3]
&\min   && \phi_{\lambda w}(x)\\
&\text{subject to} \quad && x \in \mathcal{X},
\end{aligned}\right.
\end{align}
where $\phi_{\lambda w}$ is defined by
\begin{align}\label{phi}
\phi_{\lambda w}(x)=\left(\sum_{i=1}^{p}w_{i}(s_{i}(x))^{\lambda}\right)^{\frac{1}{\lambda}},~~~x \in \mathcal{X},
\end{align}
for a given positive $\lambda > 0$ and $w_i \geq 0,~i=1,2,\ldots,p,$ with $\sum_{i=1}^{p}w_{i}=1.$ A solution to \eqref{GOWA} is called a GOWA robust optimal solution to $\mathcal{P}(\Omega).$
\end{definition}

\begin{note}\label{lambdazero}
Suppose for any given $\zeta \in \Omega,~f(x,\zeta) > 0$ for all $x\in \mathcal{X}.$ Let $\phi_{\lambda w}$ be as defined by \eqref{phi}. Then, for any $x \in \mathcal{X},$ we have 
\begin{enumerate}
    \item [(i)]  $\underset{\substack{\lambda \to 0+}}{\lim}\phi_{\lambda w}(x)= \prod_{i=1}^{p}(s_i(x))^{w_i},$ because
\begin{align*}
    \log\left(\lim_{\lambda \to 0+} \phi_{\lambda w}(x)\right) &= \lim_{\lambda \to 0+} \frac{\log\left(\sum_{i=1}^{p}w_{i}(s_i(x))^{\lambda}\right)}{\lambda}
     = \lim_{\lambda \to 0+} \frac{\sum_{i=1}^{p}w_{i}(s_i(x))^{\lambda}{\log(s_i(x))}}{\sum_{i=1}^{p}w_{i}(s_i(x))^{\lambda}}\\
    & = \frac{\sum_{i=1}^{p}w_{i}{\log(s_i(x))}}{\sum_{i=1}^{p}w_{i}}
    = \sum_{i=1}^{p}{\log(s_i(x))^{w_i}}
     = \log\left(\prod_{i=1}^{p}(s_i(x))^{w_i}\right).
\end{align*}
\item [(ii)] 
$\lim_{\lambda \to 1}\phi_{\lambda w}(x)= \sum_{i=1}^{p}w_{i}s_i(x).$
\end{enumerate}
\end{note}

For the simplicity of notations, throughout the rest of the paper, we simply use $\phi(x)$ instead of $\phi_{\lambda w}(x).$

Next, we show two instances of calculating the auxiliary functions $s_i^,$s, and the objective function $\phi$ of the GOWA robust counterpart. In the first instance, we consider smooth ${f(x,\zeta_i)}^,$s, and in the second instance, we take nonsmooth $f(x,\zeta_i)^,$s for all $\zeta_i^,$s.
%%%%%%%%%%%%%%%%%%%%%%%%%%%%%%%%%%%%%%%%

\begin{example}\label{GOWAexample_1}
For a single objective uncertain optimization problem $\mathcal{P}(\Omega),$ let the uncertainty set be $\Omega=\{\zeta_1, \zeta_2, \zeta_3, \zeta_4\}$ and the feasible set be $\mathcal{X}=[-2,2]$.
We define $f:\mathcal{X}\times\Omega \to \mathbb{R}$
by $f(x,\zeta_1)=\sin{x}+1, f(x,\zeta_2)=\tfrac{e^x}{3}, f(x,\zeta_3)=(\cos{x}+1)^2$ and $f(x,\zeta_4)=x^2.$ Graphs of $f(x,\zeta_1), f(x,\zeta_2), f(x,\zeta_3)$ and $f(x,\zeta_4)$ are drawn in Figure \ref{GOWAexample_a}. Notice that
\allowdisplaybreaks
\begin{align*}
    s_1(x)
    & = \begin{dcases}
        x^2, & ~\text{if}~~ x \in [-2,-1.283] \cup [1.409,2] \\
        (\cos{x}+1)^2, & ~\text{if}~~ x \in  [-1.283,1.174]\\
        \sin{x}+1, & ~\text{if}~~ x \in  [1.174,1.409],
    \end{dcases} \\
    s_2(x) &= \begin{dcases}
        x^2, & ~\text{if}~~ x \in [-2,-1.283]\cup [1.174,1.283] \\
        (\cos{x}+1)^2, & ~\text{if}~~ x \in  [-1.283,-0.636]\cup [1.283,1.409]\\
        \sin{x}+1, & ~\text{if}~~ x \in  [-0.636,1.174]\cup[1.409,1.781]\\
        \tfrac{e^x}{3}, & ~\text{if}~~ x \in  [1.781,2],
    \end{dcases}\\
        s_3(x)
    & = \begin{dcases}
        x^2, & ~\text{if}~~ x \in [-0.636,-0.459]\cup[0.910,1.283] \\
        (\cos{x}+1)^2, & ~\text{if}~~ x \in  [1.283,1.405]\\
        \sin{x}+1, & ~\text{if}~~ x \in  [-2,-1.889]\cup[-1.093,-0.636]\cup[1.781,2]\\
        \tfrac{e^x}{3}, & ~\text{if}~~ x \in  [-1.889,-1.093]\cup[-0.459,0.910]\cup [1.405,1.781],~~~\text{and}~~~
    \end{dcases} \\
     s_4(x)&= \begin{dcases}
        x^2, & ~\text{if}~~ x \in [-0.459,0.910] \\
        (\cos{x}+1)^2, & ~\text{if}~~ x \in  [1.405,2]\\
        \sin{x}+1, & ~\text{if}~~ x \in  [-1.889,-1.093]\\
        \tfrac{e^x}{3}, & ~\text{if}~~ x \in  [-2,-1.889]\cup[-1.093,-0.459]\cup[0.910,1.405].
    \end{dcases} 
\end{align*}
Graphs of $s_1(x),s_2(x),s_3(x)$ and $s_4(x)$ are drawn in Figure \ref{GOWAexample_b}.
Let us take $w_1=w_2=w_3=\tfrac{1}{3},w_4=0$  and $ \lambda = 1.$ Then, the function $\phi(x),$ defined by \eqref{phi}, becomes
 \begin{equation*}
\phi(x) =\left\{
        \begin{array}{ll}
            \frac{1}{3}\left(x^2+(\cos{x}+1)^2+\sin{x}+1\right), & \quad ~\text{if}~~ x \in [-2,-1.889]\cup [-1.093,-0.459]\cup[0.910,1.405] \\
            \frac{1}{3}\left(x^2+(\cos{x}+1)^2+\tfrac{e^x}{3}\right), & \quad ~\text{if}~~ x \in [-1.889,-1.093]\\
            \frac{1}{3}\left((\cos{x}+1)^2+(\sin{x}+1)+\tfrac{e^x}{3}\right), & \quad ~\text{if}~~ x \in [-0.459,0.910]\\
            \frac{1}{3}\left(x^2+\tfrac{e^x}{3}\sin{x}+1\right), & \quad ~\text{if}~~ x \in [1.405,2],\\
        \end{array}
    \right.
\end{equation*}
whose graph is shown in Figure \ref{GOWAexample1}. 
The GOWA robust counterpart for this example is $\underset{x \in [-2,2]}{\min}\phi(x).$ \\ The optimal solution set of this counterpart is $\{-1.279\}$, which is the GOWA robust optimal solution of the uncertain problem $\mathcal{P}(\Omega)$ under consideration.
\begin{figure}[H]
    \centering
    \subfloat[Objective function for different scenario\label{GOWAexample_a}]{{\includegraphics[scale=0.45]{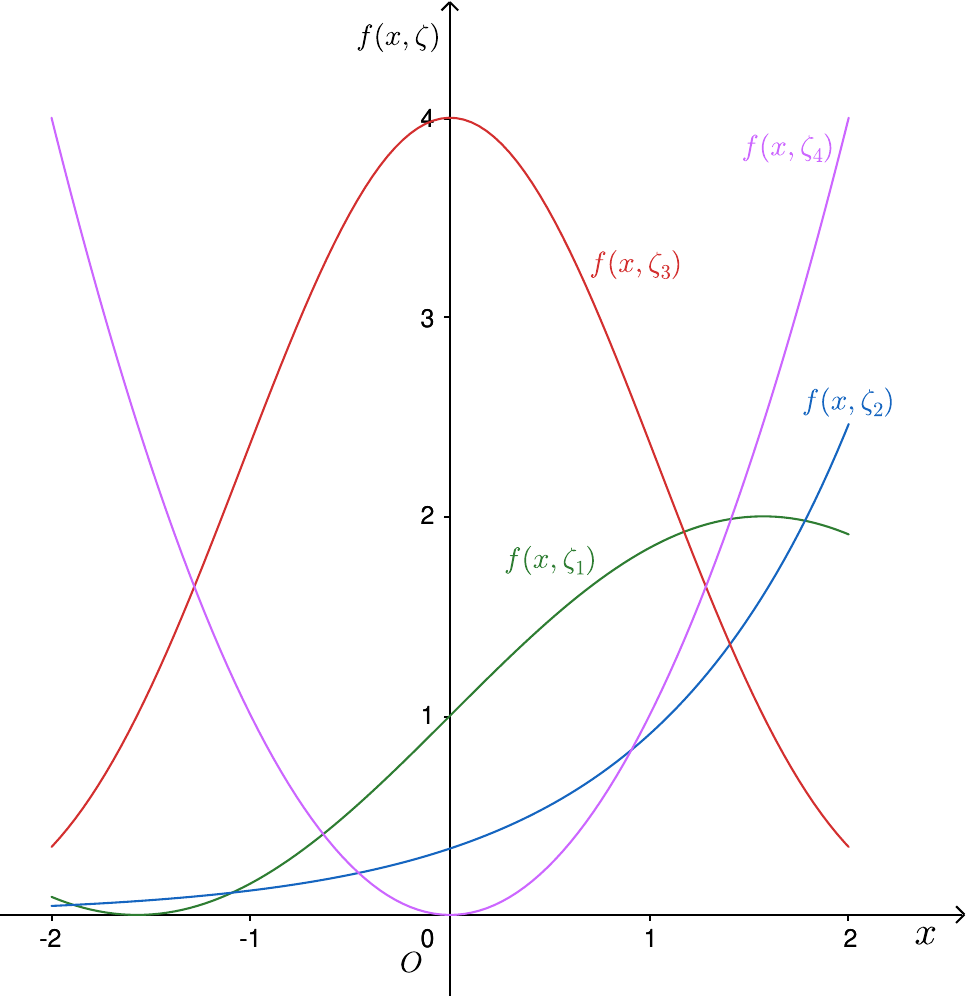} }}
    \qquad
    \subfloat[Graphs of $s_1,s_2,s_3$ and $s_4$\label{GOWAexample_b}]{{\includegraphics[scale=0.5]{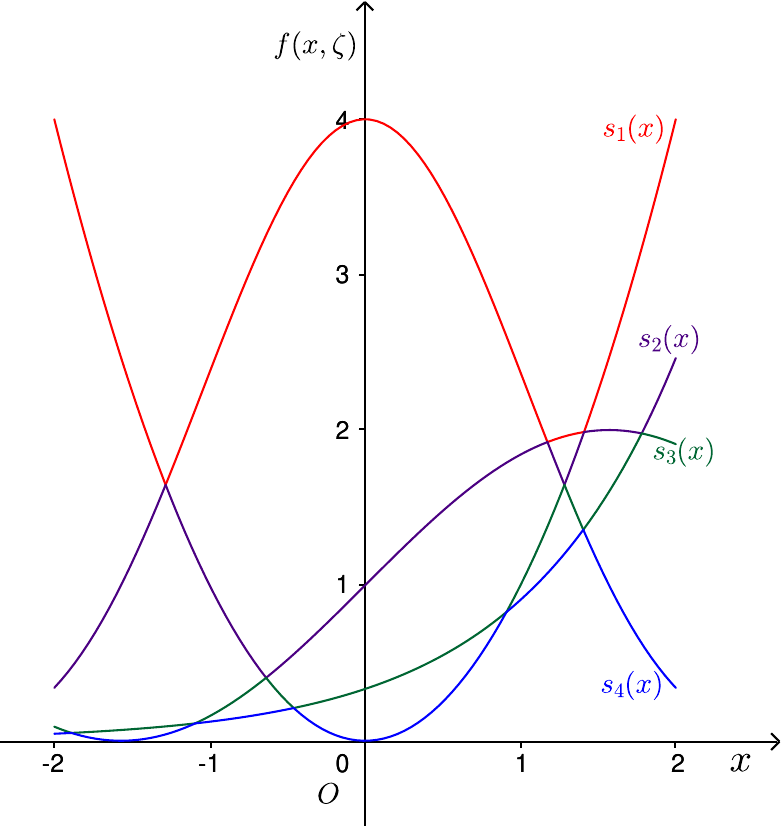} }}
    \caption{Functions involved in GOWA robust counterpart for Example \ref{GOWAexample_1}}
\end{figure}
%%%%%%%%%%%%%%%%%%%%%%%%%%%%%%%%%%%%%%%%%%%
\begin{figure}[H]
    \centering
    \includegraphics[scale=0.5]{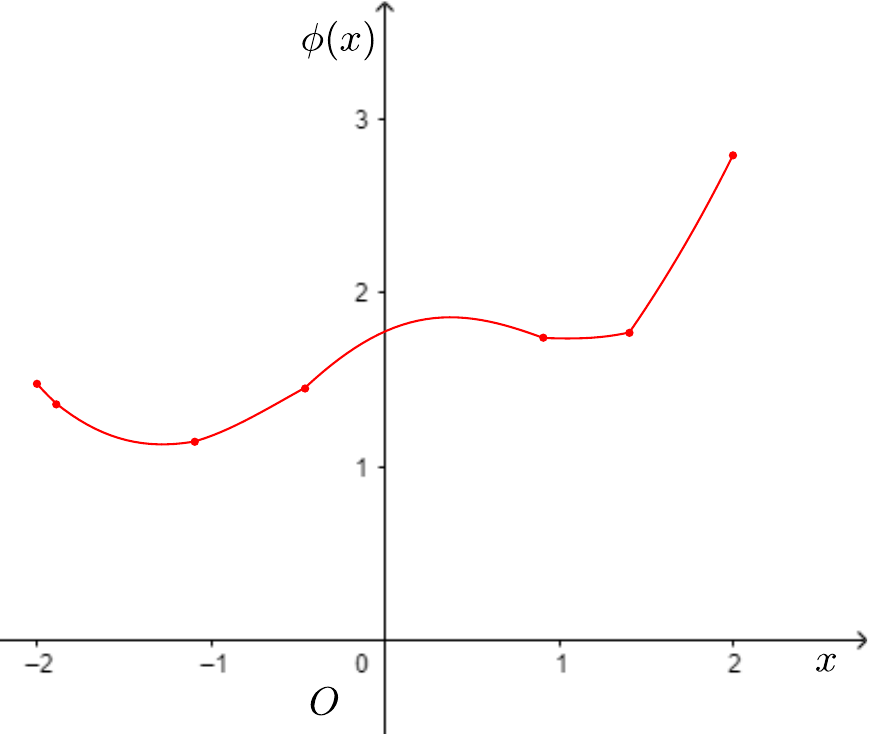}
    \caption{The function $\phi(x)$ for Example \ref{GOWAexample_1}}
    \label{GOWAexample1}
\end{figure}
\end{example}

\begin{example}\label{GOWAbestexample}
In this example, we consider continuous $f(x,\zeta_i)^,s,$, and the function $\phi$ has infinitely many nonsmooth points in any neighbourhood of a feasible point. Let the uncertainty set be $\Omega=\{\zeta_1,\zeta_2,\zeta_3\}$ and the feasible set be $\mathcal{X}=\mathbb{R}$.
Suppose $f(x,\zeta_1)=|x\sin({\tfrac{1}{x}})|,~ f(x,\zeta_2)=|x \cos({\tfrac{1}{x}})|,$ and $
f(x,\zeta_3)=|x|,~~\text{for all}~~ x \in \mathcal{X}.$ Graphs of $f(x,\zeta_1),~f(x,\zeta_2)$ and $f(x,\zeta_3)$ are drawn in Figure \ref{GOWAbestfunction}.
For $x \geq 0,$ 
\allowdisplaybreaks
\begin{align*}
    s_1(x)
    & = \begin{dcases}
        \left|x\sin{(\tfrac{1}{x})}\right|, & ~\text{if}~~ x \in \bigcup\limits_{n=1}^{\infty} \left[\tfrac{1}{(4n-1)}\tfrac{4}{\pi},\tfrac{1}{(4n-3)}\tfrac{4}{\pi} \right]\bigcup \bigcup\limits_{n=1}^{\infty} \left\{\tfrac{1}{(2n-1)}\tfrac{4}{\pi}\right\} \\
        \left|x\cos{(\tfrac{1}{x})}\right|, & ~\text{if}~~ x \in  \bigcup\limits_{n=1}^{\infty} \left[\tfrac{1}{(8n+1)}\tfrac{4}{\pi},\tfrac{1}{(8n-1)}\tfrac{4}{\pi} \right]\bigcup \bigcup\limits_{n=1}^{\infty} \left\{\tfrac{1}{(2n-1)}\tfrac{4}{\pi}\right\} \bigcup \left[\tfrac{4}{\pi}, \infty\right)\\
        \left|\tfrac{x}{\sqrt{2}}\right|, & ~\text{if}~~ x \in  \bigcup\limits_{n=1}^{\infty} \left[\tfrac{1}{(8n-4)}\tfrac{4}{\pi},\tfrac{1}{(8n-7)}\tfrac{4}{\pi} \right]\bigcup \bigcup\limits_{n=1}^{\infty} \left\{\tfrac{1}{(8n-1)}\tfrac{4}{\pi}\right\},
    \end{dcases} \\
    s_2(x)& = \begin{dcases}
        \left|x\sin{(\tfrac{1}{x})}\right|, & ~\text{if}~~ x \in \bigcup\limits_{n=1}^{\infty} \left[\tfrac{1}{(8n-6)}\tfrac{4}{\pi},\tfrac{1}{(8n-7)}\tfrac{4}{\pi} \right]\bigcup \bigcup\limits_{n=1}^{\infty} \left[\tfrac{1}{(8n-4)}\tfrac{4}{\pi},\tfrac{1}{(8n-5)}\tfrac{4}{\pi} \right] \\
        &~~~~~~~~~~~~~~~~~~~~~~~~~~\bigcup \bigcup\limits_{n=1}^{\infty} \left\{\tfrac{1}{(8n-1)}\tfrac{4}{\pi}\right\} \bigcup \bigcup\limits_{n=1}^{\infty} \left\{\tfrac{1}{(2n-1)}\tfrac{4}{\pi}\right\} \\
        \left|x\cos{(\tfrac{1}{x})}\right|, & ~\text{if}~~ x \in  \bigcup\limits_{n=1}^{\infty}\left[\tfrac{1}{(8n-5)}\tfrac{4}{\pi},\tfrac{1}{(8n-6)}\tfrac{4}{\pi} \right]\bigcup \bigcup\limits_{n=1}^{\infty} \left\{\tfrac{1}{(2n-1)}\tfrac{4}{\pi}\right\}\\
        \left|\tfrac{x}{\sqrt{2}}\right|, & ~\text{if}~~ x \in  \bigcup\limits_{n=1}^{\infty} \left[\tfrac{1}{(4n-1)}\tfrac{4}{\pi},\tfrac{1}{(4n-3)}\tfrac{4}{\pi} \right]\bigcup\bigcup\limits_{n=1}^{\infty} \left[\tfrac{1}{(8n+1)}\tfrac{4}{\pi},\tfrac{1}{(8n-1)}\tfrac{4}{\pi} \right]\\
        &~~~~~\bigcup \bigcup\limits_{n=1}^{\infty} \left[\tfrac{1}{(8n-4)}\tfrac{4}{\pi},\tfrac{1}{(8n-5)}\tfrac{4}{\pi} \right] \bigcup \bigcup\limits_{n=1}^{\infty} \left\{\tfrac{1}{(8n-3)}\tfrac{4}{\pi}\right\} \bigcup \left[\tfrac{4}{\pi}, \infty\right),~~~\text{and}\\
    \end{dcases}\\
        s_3(x)
    & = \begin{dcases}
        \left|x\sin{(\tfrac{1}{x})}\right|, & ~\text{if}~~ x \in \bigcup\limits_{n=1}^{\infty} \left[\tfrac{1}{(8n+1)}\tfrac{4}{\pi},\tfrac{1}{(8n-1)}\tfrac{4}{\pi} \right]\bigcup \bigcup\limits_{n=1}^{\infty} \left[\tfrac{1}{(8n-5)}\tfrac{4}{\pi},\tfrac{1}{(8n-6)}\tfrac{4}{\pi} \right]\\
        &~~~~~~~~~~~~~~~~~~~\bigcup \bigcup\limits_{n=1}^{\infty} \left\{\tfrac{1}{(8n-3)}\tfrac{4}{\pi}\right\}\bigcup\limits_{n=1}^{\infty} \left\{\tfrac{1}{(8n-1)}\tfrac{4}{\pi}\right\} \bigcup \left[\tfrac{4}{\pi}, \infty\right)\\
        \left|x\cos{(\tfrac{1}{x})}\right|, & ~\text{if}~~ x \in  \bigcup\limits_{n=1}^{\infty} \left[\tfrac{1}{(8n-4)}\tfrac{4}{\pi},\tfrac{1}{(8n-5)}\tfrac{4}{\pi} \right]\bigcup \bigcup\limits_{n=1}^{\infty} \left[\tfrac{1}{(8n-6)}\tfrac{4}{\pi},\tfrac{1}{(8n-7)}\tfrac{4}{\pi} \right]\\
        &~~~~~~~~~~~~~~\bigcup \bigcup\limits_{n=1}^{\infty} \left[\tfrac{1}{(4n-1)}\tfrac{4}{\pi},\tfrac{1}{(4n-3)}\tfrac{4}{\pi} \right] \bigcup \bigcup\limits_{n=1}^{\infty} \left\{\tfrac{1}{(8n-1)}\tfrac{4}{\pi}\right\}\\
        \left|\tfrac{x}{\sqrt{2}}\right|, & ~\text{if}~~ x \in  \bigcup\limits_{n=1}^{\infty} \left\{\tfrac{1}{(2n-1)}\tfrac{4}{\pi}\right\}.
    \end{dcases} 
\end{align*}
For $x <0,~ s_i(x)=s_i(-x),~ i=1,2,3.$
\\
Let us take $w_1=w_2=w_3=\tfrac{1}{3} $  and $ \lambda = 1.$ Then, the function $\phi(x),$ defined by \eqref{phi}, becomes
 \begin{equation*}
\phi(x) =\frac{1}{3}\left(\left|\tfrac{x}{\sqrt{2}}\right|+\left|x \sin{(\tfrac{1}{x})}\right|+\left|x \cos{(\tfrac{1}{x})}\right|\right), ~~ x\in \mathcal{X},
\end{equation*}
which is shown in Figure \ref{GOWAbestphi}. As the auxiliary functions $s_1, s_2$ and $s_3$ are continuous, so the objective function $\phi$ is continuous at $0$.\\
The GOWA robust counterpart for this example is \[\underset{x \in \mathbb{R}}{\min}~~\frac{1}{3}\left(\left|\tfrac{x}{\sqrt{2}}\right|+\left|x \sin{(\tfrac{1}{x})}\right|+\left|x \cos{(\tfrac{1}{x})}\right|\right).\] The optimal solution set of this counterpart is $\{0\}$, which is the GOWA robust optimal solution of the uncertain problem $\mathcal{P}(\Omega)$ under consideration.

\begin{figure}[H]
    \centering
    \includegraphics[scale=0.35]{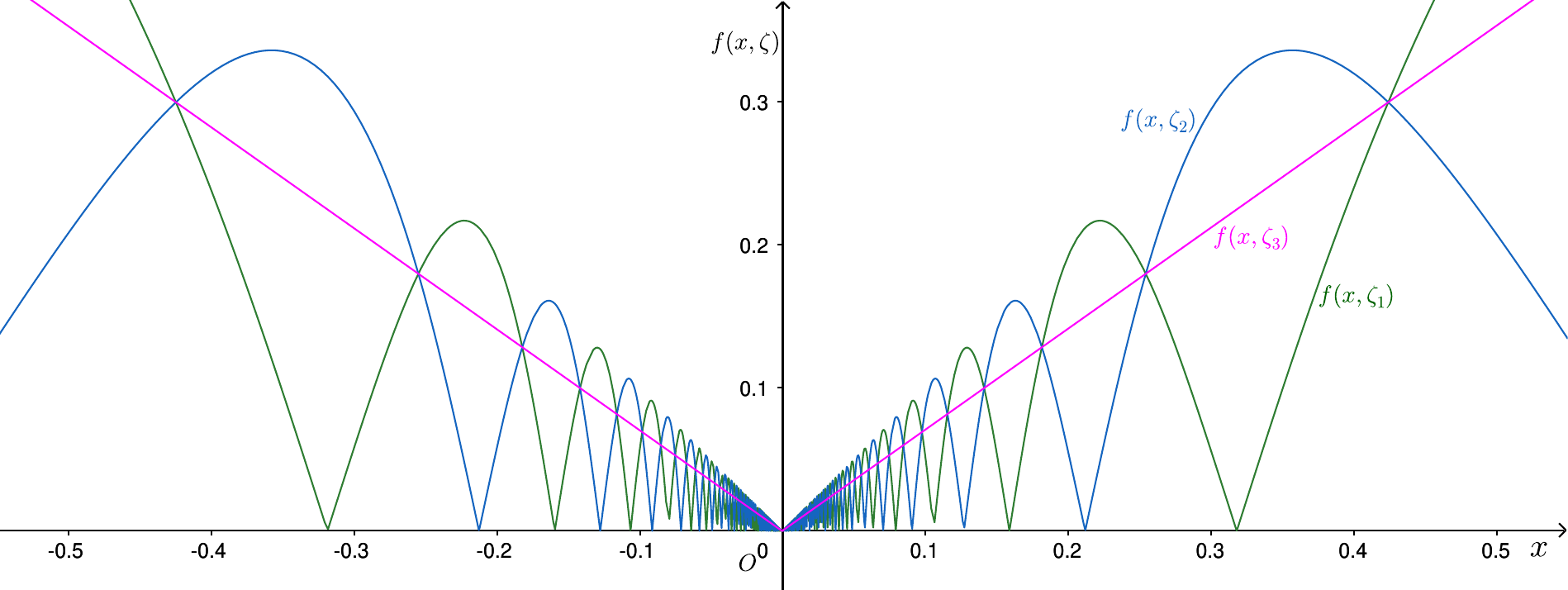}
    \caption{Objective function for different scenarios}
    \label{GOWAbestfunction}
\end{figure}
\begin{figure}[H]
    \centering
    \includegraphics[scale=0.35]{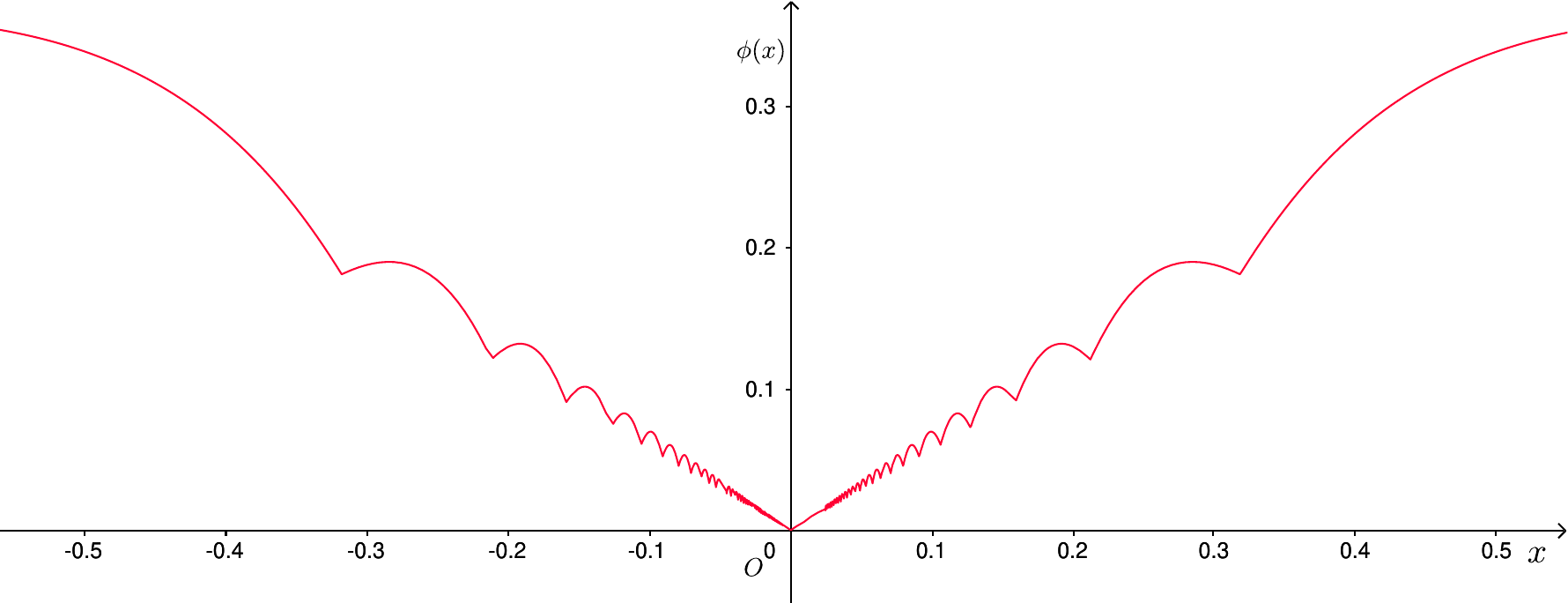}
    \caption{The function $\phi(x)$ for Example \ref{GOWAbestexample}}
    \label{GOWAbestphi}
\end{figure}
\end{example}

\noindent
We next investigate the continuity, Lipschitz continuity and coerciveness of the function $\phi(x)$ defined by \eqref{phi}. The analysis of continuity and coerciveness are helpful for the existence of an optimal solution to (\ref{GOWA}).
%%%%%%%%%%%%%%%%%%%%%%%%%%%%%%%%%%%%%%%

\noindent
\begin{theorem}\label{continuity}
\emph{{(Continuity of the function $\phi$).}}
For a given uncertain single objective optimization problem $\mathcal{P}(\Omega),$ suppose that each $f(x,\zeta_i)~i=1,2,\ldots,p,~$ is continuous on $\mathcal{X}.$ Then, the function $\phi(x)$ defined by \eqref{phi}, is continuous on $\mathcal{X}.$
\end{theorem}

\begin{proof}
According to the formulation of the functions $s_1(x)$ and $s_i(x)$, we have 
\begin{align*}
&s_1(x)= \max~\{f(x,\zeta_1),f(x,\zeta_2),\ldots,f(x,\zeta_p)\}~~~\text{and}\\
&s_i(x)= \max~\bigg\{\{f(x,\zeta_1),f(x,\zeta_2),\ldots,f(x,\zeta_p)\}\mathbin{\Big\backslash}\{s_1(x),s_2(x),\ldots,s_{i-1}(x)\}\bigg\},~~~~i=2,3,\ldots,p.
\end{align*}
As the maximum of continuous functions is continuous, therefore $s_1$ is trivially continuous.
\noindent
Next, we show that $s_i(x),$ for any $i=2,3,\ldots,p,$ is continuous at $\bar{x} \in \mathcal{X}.$

\noindent
Let $\epsilon >0.$ As $f(x,\zeta_i)$  is continuous at $\bar{x},~\text{there exists }~\delta_i$ such that \[|f(x,\zeta_i)-f(\bar{x},\zeta_i)| < \tfrac{\epsilon}{2}~~\text{ whenever }~~\|x-\bar{x}\| < \delta_i.\]

\noindent
Taking $\delta = \min \{\delta_1,\delta_2,\ldots,\delta_p\},$ we have \[|f(x,\zeta_i)-f(\bar{x},\zeta_i)| < \tfrac{\epsilon}{2} \quad \text{whenever} \quad \|x-\bar{x}\|<\delta.\] Let $x^\prime$ be such that $\|x^\prime-\bar{x}\|<\delta.$ Suppose $s_i(x^\prime)=f(x^\prime,\zeta_l)$ and $s_i(\bar{x})=f(\bar{x},\zeta_k)$ for some $l,k \in \{1,2,\ldots,p\}.$ Consider the functions $f(x,\zeta_l)$ and $f(x,\zeta_k)$ in the segment $L(x^\prime, \bar{x})$ joining $x^\prime$ and $\bar{x}.$ There can be two possible cases:

\begin{enumerate}
    \item[(i)] $f(x^\prime, \zeta_l)=f(x^\prime,\zeta_k)$ or $f(\bar{x},\zeta_l)=f(\bar{x},\zeta_k).$ In this case, if $\|x^\prime-\bar{x}\|< \delta,$ then
    \begin{align*}
    |s_i(x^\prime)-s_i(\bar{x})|&=|f(x^\prime,\zeta_k)-f(\bar{x},\zeta_k)|
    = |f(x^\prime,\zeta_l)-f(\bar{x},\zeta_l)| < \tfrac{\epsilon}{2}.
\end{align*}
   \item[(ii)] $f(x^{\prime \prime},\zeta_l)=f(x^{\prime \prime}, \zeta_k)$ for some $x^{\prime \prime} \in \text{int}~ L(x^\prime,\bar{x}),$ i.e., corresponding to each $x^\prime \in \mathcal{X}$ such that $\|x^\prime - \bar{x}\|< \delta,$ there exists $x^{\prime \prime} \in \text{int}~ L(x^\prime , \bar{x})$ such that $f(x^{\prime \prime},\zeta_l) = f(x^{\prime \prime}, \zeta_k).$ Then, 
   \begin{align*} 
   |s_i(x') - s_i(\bar x)| ~=~ & |f(x^\prime,\zeta_l) - f(\bar{x},\zeta_k)| \\ 
   ~=~ & |f(x^\prime,\zeta_l) - f(x'',\zeta_k) + f(x'', \zeta_k) -  f(\bar{x},\zeta_k)| \\ 
   ~<~ & |f(x^\prime,\zeta_l) - f(x'',\zeta_l)| + |f(x'', \zeta_k) -  f(\bar{x},\zeta_k)| < \epsilon. 
   \end{align*} 
\end{enumerate}

\noindent
For both the cases, the function $s_i(x)$ is continuous on $\mathcal{X}$ for each $i=1,2,\ldots,p.$ Hence, the function $\phi(x)$ is continuous on $\mathcal{X}.$
\end{proof}
%%%%%%%%%%%%%%%%%%%%%%%%%%%%%%%%%%%%%%%%%

\begin{remark}
Since $\phi(x)$ is continuous (Theorem \ref{continuity}) and $\mathcal{X}$ is compact, by Weierstrass theorem for extreme values, the GOWA robust counterpart \eqref{GOWA} has an optimal solution if each $f(x, \zeta_i),~i=1,2,\ldots,p,$ is continuous. That is, if each $f(x,\zeta_i),~i=1,2,\ldots,p,$ is continuous, then there exists a GOWA robust optimal solution of $\mathcal{P}(\Omega).$ 
\end{remark}

The next example illustrates that if an objective function corresponding to any scenario is discontinuous, then the function $\phi(x)$ may be discontinuous.

\begin{example}\label{GOWAdiscontinuityexample}
Let the uncertainty set be $\Omega=\{\zeta_1, \zeta_2, \zeta_3\}$ and the feasible set be $\mathcal{X}=[-2,2]$.
We define $f:\mathcal{X}\times\Omega \to \mathbb{R}$ by
$f(x,\zeta_1)= \sgn{x}+1, ~f(x,\zeta_2)= \sin{x}+1 $ and $f(x,\zeta_3)= \cos{x}+1.$ Notice that $f(x,\zeta_1)$ is discontinuous at $x=0.$ Graphs of $f(x,\zeta_1),~f(x,\zeta_2)$ and $f(x,\zeta_3)$ are drawn in Figure \ref{GOWAdiscintinuous_a}. Note that 
\begin{align*}
    s_1(x)
    & = \begin{dcases}
        \cos{x}+1, & ~\text{if}~~ x \in [-2,0] \\
        1, & ~\text{if}~~ x \in  [0,2],
    \end{dcases}\\ 
    s_2(x)& = \begin{dcases}
        \sin{x}+1, & ~\text{if}~~ x \in [-2,0] \\
        \cos{x}+1, & ~\text{if}~~ x \in  [0,0.785]\\
        \sin{x}+1, & ~\text{if}~~ x \in  [0.785,2],~~~\text{and}~~~
    \end{dcases}\\
        s_3(x)
    & = \begin{dcases}
        \sgn{x}+1, & ~\text{if}~~ x \in [-2,0] \\
        \sin{x}+1, & ~\text{if}~~ x \in  [0,0.785]\\
        \cos{x}+1, & ~\text{if}~~ x \in  [0.785, 2].
      \end{dcases}
\end{align*}
Graphs of $s_1(x),s_2(x)$ and $s_3(x)$ are drawn in Figure \ref{GOWAdiscintinuous_b}.  
Let us take $w_1=w_2=\tfrac{1}{6},~w_3=\tfrac{2}{3} $  and $ \lambda = 2.$
Then,
 \begin{equation*}
\phi(x) =\left\{
        \begin{array}{ll}
            \left(\tfrac{1}{6}(\cos{x}+1)^2+\tfrac{1}{6}(\sin{x}+1)^2+\tfrac{2}{3}(\sgn{x}+1)^2\right)^{\tfrac{1}{2}}, & \quad ~\text{if}~~ x \in [-2, 0] \\
            \left(\tfrac{1}{6}+\tfrac{1}{6}(\cos{x}+1)^2+\tfrac{2}{3}(\sin{x}+1)^2\right)^{\tfrac{1}{2}}, & \quad ~\text{if}~~ x \in [0, 0.785]\\
            \left(\tfrac{1}{6}+\tfrac{1}{6}(\sin{x}+1)^2+\tfrac{2}{3}(\cos{x}+1)^2\right)^{\tfrac{1}{2}}, & \quad ~\text{if}~~ x \in [0.785, 2],
        \end{array}
    \right.
\end{equation*}
which is shown in Figure \ref{GOWAdiscontinuity}. Clearly, we see from Figure \ref{GOWAdiscontinuity} that the function $\phi(x)$ is not continuous at $0.$
\begin{figure}[H]
    \centering
    \subfloat[The objective functions for different scenarios\label{GOWAdiscintinuous_a}]{{\includegraphics[width=6cm]{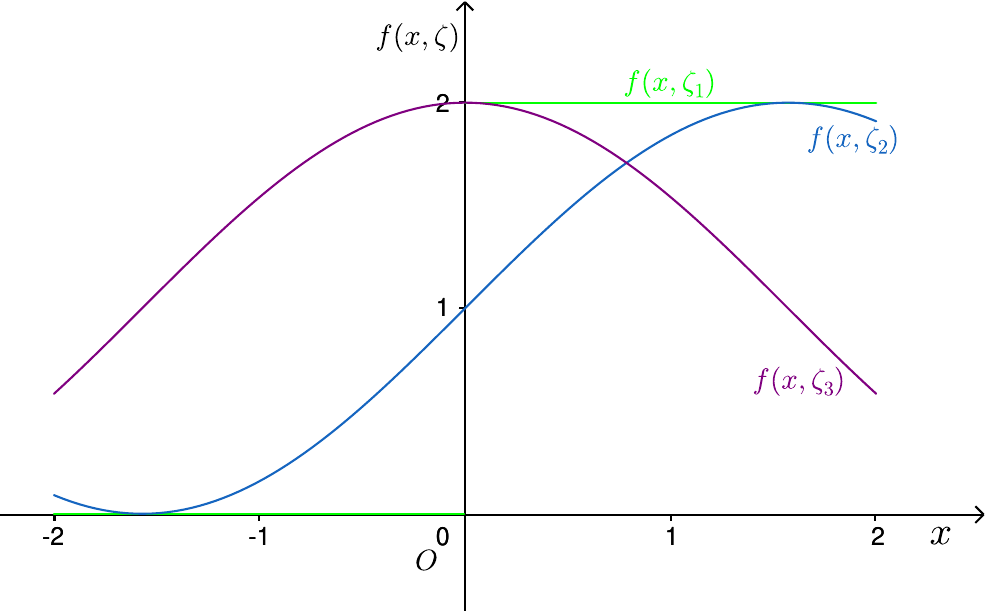} }}
    \qquad
    \subfloat[The functions $s_1, s_2, $ and $s_3$\label{GOWAdiscintinuous_b}]{{\includegraphics[width=6cm]{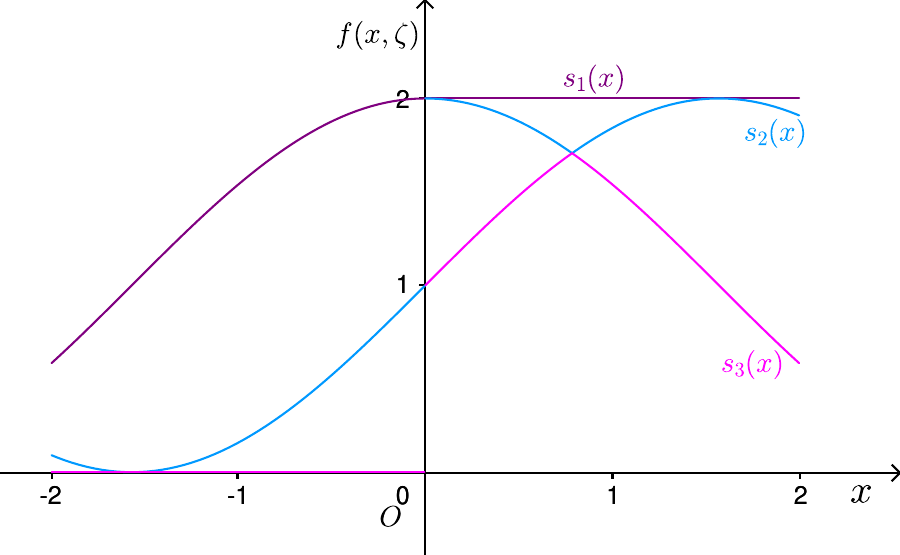} }}
    \caption{The objective function and $s_i^,$s in Example \ref{GOWAdiscontinuityexample}}
\end{figure}
\begin{figure}[H]
    \centering
    \includegraphics[scale=0.35]{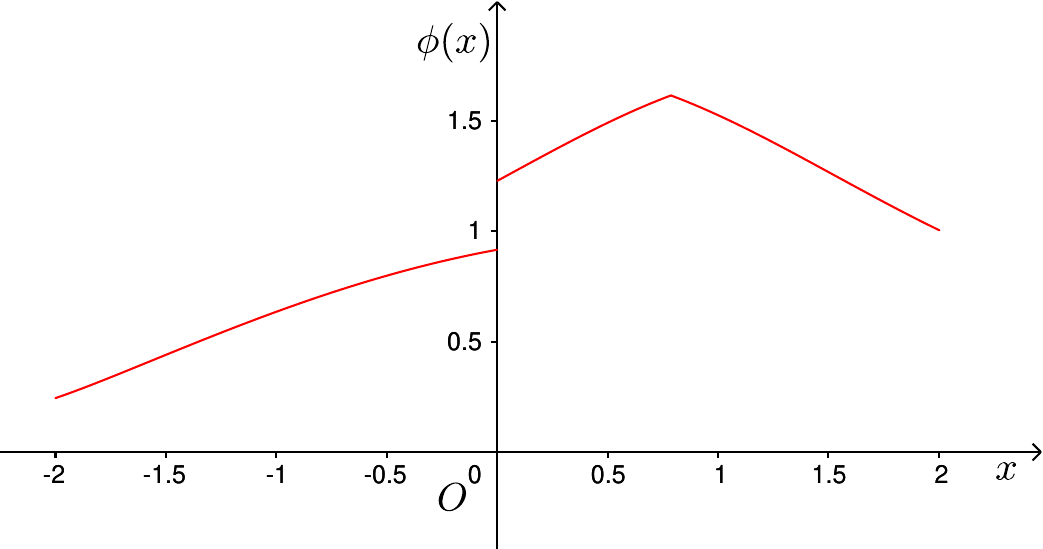}
    \caption{The function $\phi(x)$ for Example \ref{GOWAdiscontinuityexample}}
    \label{GOWAdiscontinuity}
\end{figure}
\end{example}

\begin{theorem}\label{Lipschitzcontinuity}
\emph{{(Local Lipschitz continuity of the function $\phi$).}}
For a given uncertain single objective optimization problem $\mathcal{P}(\Omega)$ if each $f(x,\zeta_i),~i=1,2,\ldots,p,~$ is locally Lipschitz continuous on $\mathcal{X},$ then the function $\phi(x)$, which is defined by \eqref{phi}, is locally Lipschitz continuous on $x \in \mathcal{X}$.
\end{theorem}

\begin{proof}
According to the formulation of the functions $s_1(x)$ and $s_i(x),$ we have 
\begin{align*}
&s_1(x)= \max~\{f(x,\zeta_1),f(x,\zeta_2),\ldots,f(x,\zeta_p)\}~~~\text{and}\\
&s_i(x)= \max~\bigg\{\{f(x,\zeta_1),f(x,\zeta_2),\ldots,f(x,\zeta_p)\}\mathbin{\Big\backslash}\{s_1(x),s_2(x),\ldots,s_{i-1}(x)\}\bigg\},~~~~i=2,3,\ldots,p.
\end{align*}

\noindent
As the maximum of locally Lipschitz continuous functions is locally Lipschitz continuous, $s_1$ is trivially locally Lipschitz continuous.

\noindent
Next, we show that the functions $s_i(x),$ for any $i=2,3,\ldots,p,$ is locally Lipschitz continuous at $\bar{x} \in \mathcal{X}.$
As $f(x,\zeta_i)$  is locally Lipschitz continuous at $\bar{x},~\text{there exists} ~\Gamma_i >0$ and $\delta_i >0$ such that \[|f(x,\zeta_i)-f(\bar{x},\zeta_i)| \leq \Gamma_i \|x-\bar{x}\|~~\text{whenever}~~\|x-\bar{x}\| < \delta_i.\]
Taking $\delta = \min~ \{\delta_1,\delta_2,\ldots,\delta_p\}$ and $\Gamma = \max ~\{\Gamma_1, \Gamma_2, \ldots, \Gamma_p\},$ we have \[|f(x,\zeta_i)-f(\bar{x},\zeta_i)| \leq \Gamma \|x-\bar{x}\|~~\text{whenever}~~ \|x-\bar{x}\|<\delta ~\text{ for all }~ i=1,2,\ldots,p.\] Let $x^\prime$ be such that $\|x^\prime-\bar{x}\|<\delta.$ Suppose $s_i(x^\prime)=f(x^\prime,\zeta_l)$ and $s_i(\bar{x})=f(\bar{x},\zeta_k)$ for some $l,k \in \{1,2,\ldots,p\}.$ Consider the curves $f(x,\zeta_l)$ and $f(x,\zeta_k)$ in the line segment $L(x^\prime, \bar{x})$ joining $x^{\prime}$ and $\bar{x}.$ There can be two possible cases:
\begin{enumerate}
    \item[(i)] \label{case1} $f(x^\prime, \zeta_l)=f(x^\prime,\zeta_k)$ or $f(\bar{x},\zeta_l)=f(\bar{x},\zeta_k).$ In this case, if $\|x^\prime-\bar{x}\|< \delta,$ then
\begin{align*}
    |s_i(x^\prime)-s_i(\bar{x})|&=|f(x^\prime,\zeta_l)-f(\bar{x},\zeta_k)|\\
    &= |f(x^\prime,\zeta_k)-f(\bar{x},\zeta_k)|
    = |f(x^\prime,\zeta_l)-f(\bar{x},\zeta_l)| \leq \Gamma \|x^\prime-\bar{x}\|
\end{align*}

     \item[(ii)] $f(x^{\prime \prime},\zeta_l)=f(x^{\prime \prime}, \zeta_k)$ for some $x^{\prime \prime} \in \text{int}~ L(x^\prime,\bar{x}),$ i.e., corresponding to each $x^\prime \in \mathcal{X}$ such that $\|x^\prime - \bar{x}\|< \delta,$ there exist $x^{\prime \prime} \in \text{int}~ L(x^\prime , \bar{x})$ such that $f(x^{\prime \prime},\zeta_l) = f(x^{\prime \prime}, \zeta_k).$ Then, 
   \begin{align*} 
   |s_i(x') - s_i(\bar x)| ~=~ & |f(x^\prime,\zeta_l) - f(\bar{x},\zeta_k)| \\ 
   ~=~ & |f(x^\prime,\zeta_l) - f(x'',\zeta_k) + f(x'', \zeta_k) -  f(\bar{x},\zeta_k)| \\ 
   ~<~ & |f(x^\prime,\zeta_l) - f(x'',\zeta_l)| + |f(x'', \zeta_k) -  f(\bar{x},\zeta_k)| \leq 2 \Gamma \|x^\prime-\bar{x}\|. 
   \end{align*} 
\end{enumerate}
For both the cases, the function $s_i(x)$ is locally Lipschtiz continuous on $\mathcal{X}$ for each $i=1,2,\ldots,p.$ Hence, the function $\phi(x)$ is locally Lipschitz continuous on $\mathcal{X}.$
\end{proof}

\begin{theorem}\label{coeirciveness}
\emph{(Coerciveness of the function  $\phi$).}
For a given uncertain single objective optimization problem $\mathcal{P}(\Omega),$ suppose each $f(x,\zeta_i),~ i=1,2,\ldots,p,$ is a coercive function on $\mathbb{R}^n.$ Then, the function $\phi(x)$, defined by \eqref{phi}, is a coercive function on $\mathbb{R}^n$.
\end{theorem}

\begin{proof}
Since each $f(x,\zeta_i),~ i=1,2,\ldots,p,$ is coercive, i.e., $f(x,\zeta_i)$ goes to infinity as $\norm{x}$ goes to infinity, the maximum function of $f(x,\zeta_i)$ also goes to infinity as $\norm{x}$ goes to infinity. Therefore, $s_1(x)$ is coercive. Next, we show that $\underset{~~~\|x\| \to \infty}{\lim}~s_i(x)= +\infty.$ Let $K >0$. Since $\underset{\|x\| \to \infty}{\lim}f(x,\zeta_i)= +\infty$ for all $i=1,2,\ldots,p,$ there exists $G >0$ such that for any $i \in \{1,2,\ldots,p\},$ we have
\[|f(x, \zeta_i)| > K,~\text{for all}~ x \in \mathbb{R}^n~~\text{with}~~ \|x\| > G.\]
Hence, by the definition of $s_i(x),$
\[|s_i(x)| >K,~ \text{for all}~ x \in \mathbb{R}^n ~~\text{with}~~ \|x\| >G.\] Thus, $\underset{\|x\| \to \infty}{\lim}s_i(x)= +\infty$ for all $i=1,2,\ldots,p,$ and hence, $\phi$ is a coercive function on $\mathbb{R}^n.$
\end{proof}

\begin{theorem}\emph{(Monotonicity of the function $\phi$).}\label{monotonicity_of_phi}
Let $\mathcal{X}$ be a nonempty compact subset of $\mathbb{R}.$ Consider an uncertain single objective optimization problem $\mathcal{P}(\Omega).$ 
If the functions $f(x,\zeta_i),~i=1,2,\ldots,p,$ are increasing on $\mathcal{X},$ then the function $\phi(x),$ as defined by \eqref{phi}, is increasing on $\mathcal{X}.$
\end{theorem}

\begin{proof}
First, we show that each auxiliary function $s_i(x), i=1,2,\ldots,p,$ is an increasing function on $\mathcal{X}$. Let $x_1,~ x_2 \in \mathcal{X}$ be such that $x_1 \geq x_2.$ Then, $f(x_1, \zeta_i) \geq f(x_2, \zeta_i),~ i=1,2,\ldots,p.$ By taking maximum on both sides over each scenario, we get $\underset{\zeta_i \in \Omega}{\max}~f(x_1, \zeta_i) \geq \underset{\zeta_i \in \Omega}{\max}~f(x_2, \zeta_i)$ which implies that $s_1(x_1) \geq s_1(x_2).$ This shows that $s_1(x)$ is an increasing function on $\mathcal{X}$. Let $s_1(x_1)=f(x_1, \zeta_j) ~\text{and}~s_1(x_2)=f(x_2, \zeta_k), ~\text{ for some}~ j, k \in \Omega.$ Next, we show that $s_2(x)$ is an increasing function on $\mathcal{X}$. On the contrary, suppose that $s_2(x)$ is not an increasing function. Then, there exist $y_1,y_2$ with $y_1 > y_2$ such that $s_2(y_1) < s_2(y_2),$
\begin{align*}
    &\text{ i.e., }~\max\{\{f(y_1, \zeta_1),f(y_1, \zeta_2), \ldots, f(y_1, \zeta_p)\} \setminus s_1(y_1)\} \\& \hspace{7cm}< \max\{\{f(y_2, \zeta_1),f(y_2, \zeta_2), \ldots, f(y_2, \zeta_p)\} \setminus s_1(y_2)\},\\
    &\text{ i.e., }~ f(y_1, \zeta_m) <  f(y_2, \zeta_n), \text{ for some } m,n \in \Omega \text{ and } n \neq j, m \neq k,\\
    &\text{ i.e., }~ f(y_1, \zeta_n) \leq f(y_1, \zeta_m) < f(y_2, \zeta_n),\\
    &\text{ i.e., }~ f(y_1, \zeta_n) < f(y_2, \zeta_n).
\end{align*}
This implies that $f(x, \zeta_n)$ is an increasing function, which is contradictory. Therefore, $s_2(x)$ is an increasing function on $\mathcal{X}$. In a similar way, we can show that each auxiliary function $s_i(x),~ i=3,4,\ldots,p,$ is also increasing on $\mathcal{X}.$ Hence, the function $\phi(x),$ defined by \eqref{phi}, is an increasing function on $\mathcal{X}.$
\end{proof}

\begin{corollary}
Let $\mathcal{X}$ be a nonempty compact subset of $\mathbb{R}.$ Consider an uncertain single objective optimization problem $\mathcal{P}(\Omega).$ If the functions $f(x,\zeta_i),~i=1,2,\ldots,p,$ all are decreasing on $\mathcal{X},$ then the function $\phi(x),$ as defined by \eqref{phi}, is decreasing on $\mathcal{X}.$
\end{corollary}

\begin{definition} (Increasing function along a direction).
     A ray at $\bar{x}$ in the direction $d$ is the set $R(\bar{x})= \{x \in \mathbb{R}^n : x=\bar{x}+td, ~t \in \mathbb{R}\}.$ A function $\psi : \mathbb{R}^n \to \mathbb{R}$ is increasing along a direction $d$ on the ray $R(\bar{x})$ if for any $x_1, x_2 \in \mathbb{R}^n,$ we have $x_1 \geq x_2$ if and only if $t_1 \geq t_2,$ where $x_i=\bar{x}+t_id,~ i=1,2.$
\end{definition}

\begin{corollary}
   Let $\mathcal{X}$ be a compact subset of $\mathbb{R}^n,~n\geq 2.$ For a given uncertain single objective optimization problem $\mathcal{P}(\Omega)$ if $f(x, \zeta_i), i=1,2,\ldots,p,$ are all increasing along a direction $d \in \mathbb{R}^n$ emanated from $\bar{x} \in \text{int}~\mathcal{X},$ then $\phi$ is an increasing function along the direction $d$. 
\end{corollary}

\begin{proof}
    As each function $f(x,\zeta_i),~ i=1,2,\ldots,p,$ is increasing on $R(\bar{x})= \{x \in \mathbb{R}^n : x=\bar{x}+td, ~t \in \mathbb{R}\},$ and $R(\bar{x})$ is homeomarphic to the real line $\mathbb{R},$ by Theorem \ref{monotonicity_of_phi}, the function $\phi$ is increasing along $d$ at $\bar{x}.$ 
\end{proof}

\begin{example}\label{example_for_nonmonotonus}
 (Example for nonmonotonous $\phi(x)$).
Let the uncertainty set be $\Omega=\{0,\frac{\pi}{2}\},$ and the feasible set be $\mathcal{X}=[0,1].$ We define $f:\mathcal{X}\times\Omega \to \mathbb{R}$ by $f(x,\zeta)= \cos(\zeta)x^3+\sin(\zeta)(1-x^2).$\\
For $\zeta_1=\pi/2,~f(x,\zeta_1)=x^3,$ and for $\zeta_2=\pi,~f(x,\zeta_2)=1-x^2.$ 

\noindent
Graphs of $f(x,\zeta_1)$ and $f(x,\zeta_2)$ are drawn in Figure \ref{nonmono}.

\noindent
Note that
\begin{align*}
    s_1(x)
    & = \begin{dcases}
        1-x^2, & ~\text{if}~ x \in [0,0.755] \\
        x^3, & ~\text{if}~ x \in  [0.755,1],~~~~~ \text{and} ~~~
    \end{dcases}\\
    s_2(x)
     &= \begin{dcases}
        x^3, & ~\text{if}~ x \in [0,0.755] \\
        1-x^2, & ~\text{if}~ x \in  [0.755,1]. 
    \end{dcases} 
\end{align*}
Let us take $w_1=\frac{1}{3}, $ $ w_2=\frac{2}{3} $  and $ \lambda = 2.$ Then,
\begin{align*}
    \phi(x)& = \begin{dcases}
        \left(\tfrac{1}{3}(1-2x^2+x^4+2x^6)\right)^{\tfrac{1}{2}}, &~\text{if}~ x \in [0,0.755] \\
        \left(\tfrac{1}{3}(2-4x^2+2x^4+x^6)\right)^{\tfrac{1}{2}}, &~\text{if}~ x \in  [0.755,1]. 
    \end{dcases}
\end{align*}
Hence, we see that $f(x,\zeta_1)$ is increasing and $f(x,\zeta_2)$ is decreasing on $\mathcal{X},$ but $\phi(x)$ is neither increasing nor decreasing (shown in Figure \ref{nonmono}).
\begin{figure}[H]
    \centering
    \includegraphics[scale=0.4]{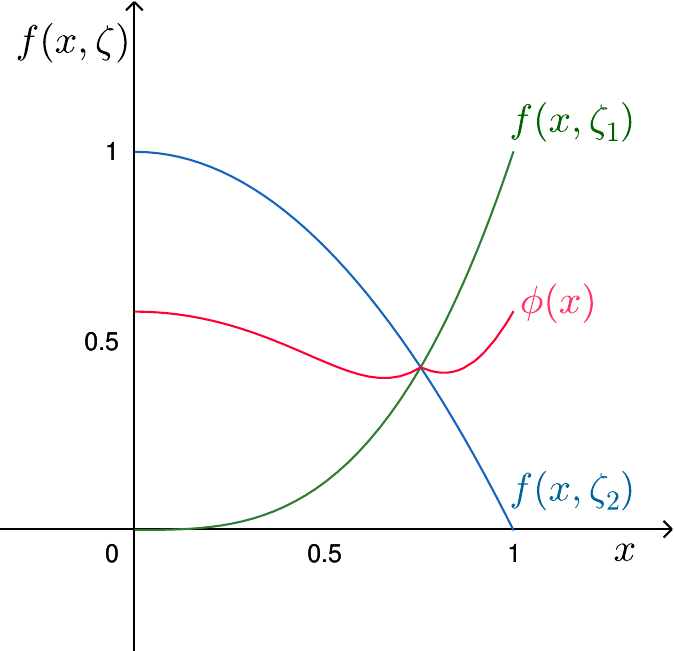}
    \caption{The functions that contribute to Example \ref{example_for_nonmonotonus}}
    \label{nonmono}
\end{figure}
\end{example}
%%%%%%%%%%%%%%%%%%%%%%%%%%%%%%%%%%%%%%%%%%%%%%%%

\section{Relation of GOWA robustness with other robustness}\label{relation}

\subsection{Min-max robust efficiency}
One of the widely used concepts of robustness for uncertain single objective optimization problems is the concept of min-max robust optimality, originally introduced by Soyster \cite{soyster} and extensively studied in \cite{Ben}.\\
 A geometrical interpretation of min-max robustness is given in Figure \ref{min-max_figure}. For instance, let us take $p=4,$ and $\mathcal{X}=[0,\zeta].$ Graphs of the objective functions for different scenarios are drawn in Figure \ref{min-max_figure}. First, we take the maximum of all four objective functions, as shown by the dotted curve, and then minimize it on $\mathcal{X}.$ The minimum point of this maximum function (dotted graph) is $\bar{x},$ which is the min-max robust solution.  
\begin{figure}[H]
    \centering
    \includegraphics[scale=0.65]{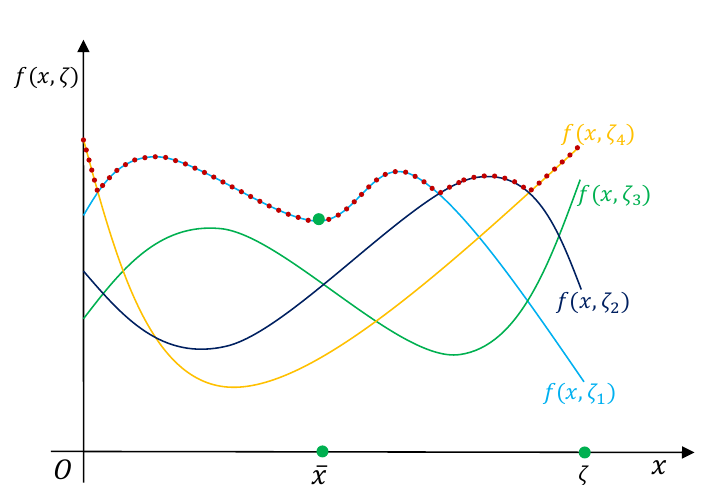}
    \caption{The functions that contribute to min-max robustness}
    \label{min-max_figure}
\end{figure}

\begin{proposition}
\emph{(Relation with min-max robustness).}\label{min-maxthm}
For a given uncertain single objective optimization problem $\mathcal{P}(\Omega),$ let $w_1=1,~w_i=0, ~i=2,3,\ldots,p$~~and~~$\lambda=1$. Then, a GOWA robust optimal solution is a min-max robust optimal solution and vice-versa.
\end{proposition}

\begin{proof}
Since $w_1=1,$ $w_i=0$ $\text{ for all }$ $i=2,3,\ldots,p$, by \eqref{phi} we have
$$\phi(x)=s_1(x)= \max~\{f(x,\zeta_1),f(x,\zeta_2),\ldots,f(x,\zeta_p)\}.$$
This expression of $\phi$ trivially implies the result.
\end{proof}

 In the next example, we show that not for any chosen $w_i^,$s and $\lambda,$ the set of GOWA robust optimum is identical to the set of min-max robust optimum.

\begin{example}\label{gowa_optimum_is_not_equal_to_min_max_optimum}
Let the uncertainty set be $\Omega = \{\pi/2,\pi\}$ and the feasible set be $\mathcal{X}=[-1,1].$ We define $f:\mathcal{X}\times \Omega \to \mathbb{R}$ by $f(x,\zeta)=x\cos(\zeta)+x^3\sin(\zeta)+1.$\\
For $\zeta_1=\pi/2,~f(x,\zeta_1)=x^3+1$
and for $\zeta_2=\pi,~f(x,\zeta_2)=-x+1.$ Graphs of $f(x,\zeta_1)$ and $f(x,\zeta_2)$ are drawn in Figure \ref{min-maxgowaexample_1}.

\noindent
Let us take $w_1=\tfrac{3}{10},~ w_2=\tfrac{7}{10} $  and $ \lambda = \tfrac{1}{2}.$

\noindent
Then, the min-max robust objective function is
 \begin{equation*}
\bar{\phi}(x)=\underset{\zeta\in \Omega}{\max}~ f(x,\zeta)=\left\{
        \begin{array}{ll}
            -x+1, & \quad ~~~\text{if}~~x \in [-1, 0] \\
            x^3+1, & \quad~~~\text{if}~~ x \in [0, 1].
        \end{array}
    \right.
\end{equation*}
The GOWA robust objective function is
 \begin{equation*}
\phi(x)=\left(w_1(b_1(x))^{\tfrac{1}{2}}+w_2(b_2(x))^{\tfrac{1}{2}}\right)^{2}=\left\{
        \begin{array}{ll}
           \left(\tfrac{3}{10}(-x+1)^{\tfrac{1}{2}}+\tfrac{7}{10}(x^3+1)^{\tfrac{1}{2}}\right)^2, & \quad~~~\text{if}~~ x \in [-1, 0] \\
            \left(\tfrac{3}{10}(x^3+1)^{\tfrac{1}{2}}+\tfrac{7}{10}(-x+1)^{\tfrac{1}{2}}\right)^2, & \quad~~~\text{if}~~ x \in [0, 1].
        \end{array}
    \right.
\end{equation*}
Graphs of $\bar{\phi}(x)$ and $\phi(x)$ are drawn in the Figure \ref{min-maxgowaexample_2}. From Figure \ref{min-maxgowaexample_2}, we notice that the set of min-max robust optimal solution of $\mathcal{P}(\Omega)$ is $\{0\},$ and the set of GOWA robust optimal solution of $\mathcal{P}(\Omega)$ is $\{-1,1\}.$
\begin{figure}[H]
    \centering
    \subfloat[The objective functions for different scenarios\label{min-maxgowaexample_1}]{\includegraphics[width=5cm]{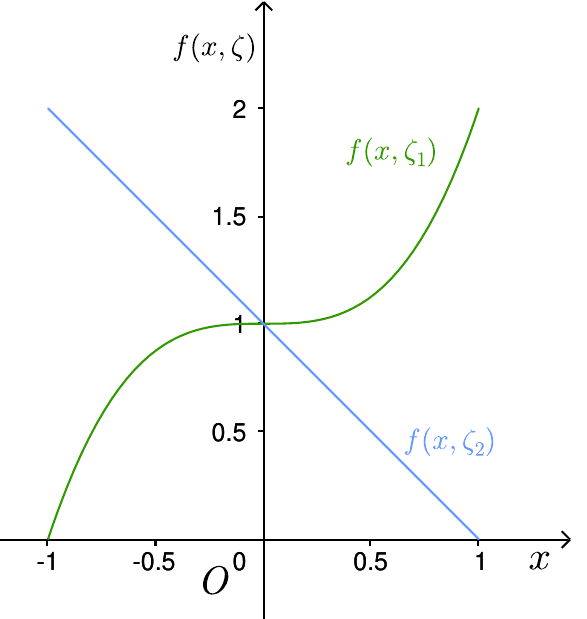} }
    \qquad \qquad 
    \subfloat[The function $\phi(x)$\label{min-maxgowaexample_2}]{{\includegraphics[width=5cm]{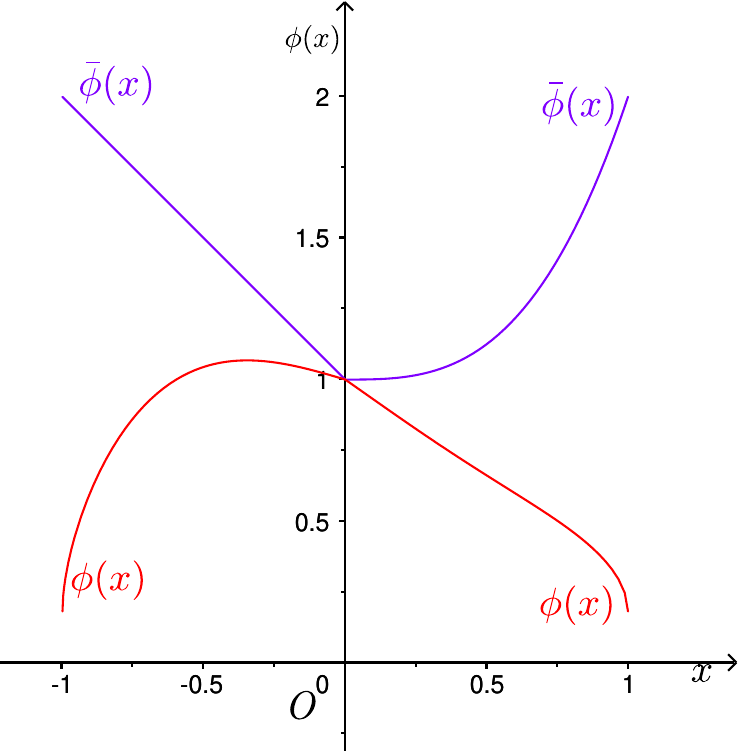} }}
    \caption{The functions that contribute to Example \ref{gowa_optimum_is_not_equal_to_min_max_optimum}}
\end{figure}
\end{example}
\subsection{Light robust optimality}\label{light_robust_optimality}
\begin{definition} \cite{Robustness}
For a given uncertain single objective optimization problem $\mathcal{P}(\Omega)$, assume that $\hat{x}$ is an optimal solution to the optimization problem $\mathcal{P}(\hat{\zeta})$ of the nominal scenario $\hat{\zeta}.$ Then, an element $\bar{x} \in \mathcal{X}$ is called a lightly robust optimal solution to $\mathcal{P}(\Omega)$ with respect to a given $\epsilon \geq 0$  if it is an optimal solution to
 \begin{align} \label{lightrobust}
\left.\begin{aligned}%[3]
&\min  &&\underset{\zeta\in \Omega}{\max}~ f(x,\zeta)\\
&\text{subject to} \quad && f(x,\hat{\zeta})\leq f(\hat{x},\hat{\zeta})+\epsilon\\
&\quad && x \in \mathcal{X}.
\end{aligned}\right.
\end{align}
\end{definition}
The geometrical interpretation of light robust optimality is shown in Figure \ref{light}. For instance, let us take $p=4.$ Graphs of the objective functions for different scenarios are drawn in Figure \ref{light}. For a given $\epsilon >0,$ by taking an additional condition $f(x,\hat{\zeta}) \leq f(\hat{x},\hat{\zeta})+\epsilon$ on $\mathcal{X},$ we take the maximum of all four objective functions, as shown by the dotted curve, and then minimize it on $\mathcal{X}.$  
\begin{figure}[H]
    \centering
    \includegraphics[scale=0.60]{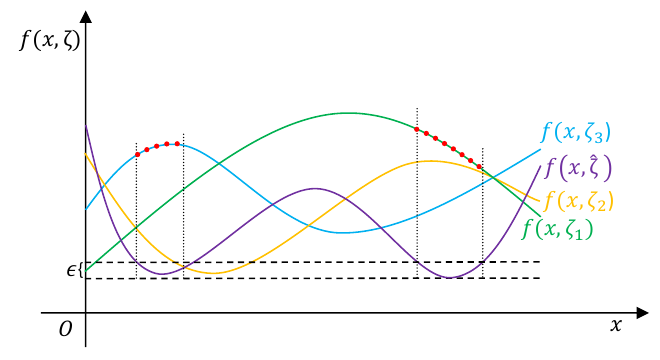}
    \caption{Functions involved in light robustness}
    \label{light}
\end{figure}

\begin{example}
     In Example \ref{GOWAexample_1}, let the nominal scenario be $\hat{\zeta}=\zeta_1$ and $\epsilon = \tfrac{1}{2}$. The minimum solution of $f(x, \hat{\zeta})=\sin(x)+1$ on $\mathcal{X}$ is $\hat{x}=-1.570$ and optimal value is $f(\hat{x},\hat{\zeta})=0$. Thus, a lightly robust optimal solution set to $\mathcal{P}(\Omega)$ will be obtained by solving the following problem:

\begin{equation*}
\begin{aligned}
& \min
& & \underset{\zeta\in \Omega}{\max}~ f(x,\zeta) \\
& \text{subject to}
& & \sin{x}+1 \leq \tfrac{1}{2} \\
&&& x \in \mathcal{X}.
\end{aligned}
\end{equation*}
    
which is $\{-1.283\}$. Hence, for any chosen $w_i$ and $\epsilon >0$, a lightly robust optimal solution is not necessarily equal to a GOWA robust optimal solution.
\end{example}

\subsection{Min-min robust efficiency}
\begin{definition} \cite{ide}
For a given uncertain single objective optimization problem $\mathcal{P}(\Omega)$, an element $\bar{x} \in \mathcal{X}$ is called a min-min robust optimal solution to $\mathcal{P}(\Omega)$ if it is an optimal solution to 
 \begin{align}\label{min-min}
\left.\begin{aligned}%[3]
&\min   &&\underset{\zeta\in \Omega}{\min}~ f(x,\zeta)\\
&\text{subject to} \quad && x \in \mathcal{X}.\\
\end{aligned}\right.
\end{align}
\end{definition}
A geometrical illustration of min-min robustness is shown in Figure \ref{min-minfig}. For instance, let us take $p=4.$ Graphs of the objective functions for different scenarios are drawn in Figure \ref{min-minfig}. First, we take the minimum of all four objective functions, as shown by the dotted curve, and then minimize it over $\mathcal{X}.$
\begin{figure}[H]
    \centering
    \includegraphics[scale=0.6]{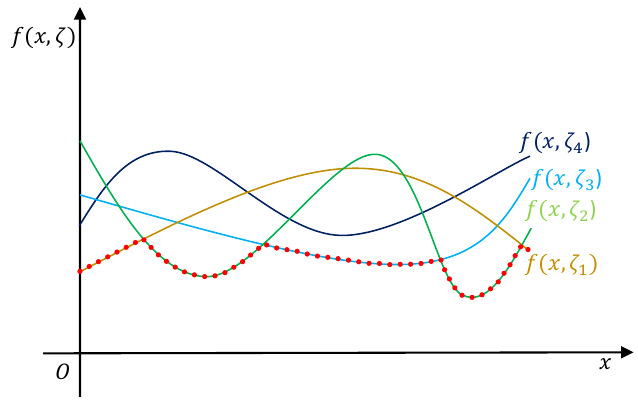}
    \caption{Functions involved in min-min robustness}
    \label{min-minfig}
\end{figure}

\begin{proposition}
\emph{(Relation with min-min robustness).}\label{min-minthm}
For a given uncertain single objective optimization problem $\mathcal{P}(\Omega),$ let $w_p=1,~w_i=0, ~i=1,2,\ldots,p-1$~~and~~$\lambda=1$. Then, a GOWA robust optimal solution is a min-min robust optimal solution and vice-versa.
\end{proposition}

\begin{proof}
Let $\bar{x}$ be the GOWA robust optimal solution of the problem $\mathcal{P}(\Omega)$. Then, $\bar{x}$ is a solution of \eqref{GOWA}.
Since $w_p=1,$ $w_i=0$ for all $i=1,2,\ldots,p-1$, therefore by \eqref{phi}, we have
$$\phi(x)=s_p(x)= \min~\{f(x,\zeta_1),f(x,\zeta_2),\ldots,f(x,\zeta_p)\}.$$ 
Hence, for all $x\in\mathcal{X}$, we can write $\phi(x)$ as a minimization problem given by
 \begin{align}\label{minx}
\phi(x)=\min_{\zeta \in \Omega}~f(x,\zeta).
\end{align}
By using \eqref{GOWA} and \eqref{minx}  , $\bar{x}$ is a min-min robust optimal solution of $\mathcal{P}(\Omega).$

\noindent

Next, let $\bar{x}$ be the solution of the problem \eqref{min-min}. In \eqref{min-min}, we represent $\phi(x)$ as a minimization problem given by \eqref{minx}. Since $w_p=1,$ $w_i=0$ $ ~\text{ for all }~$ $i=1,2,\ldots,p-1$, therefore by \eqref{phi}, we have
\begin{align}\label{b_p(x_1)}
\phi(x)=s_p(x)= \min~\{f(x,\zeta_1),f(x,\zeta_2),\ldots,f(x,\zeta_p)\}.
\end{align}
Thus, using \eqref{min-max} and \eqref{b_p(x_1)}, $\bar{x}$ is a GOWA robust optimal solution of the uncertain single objective optimization problem $\mathcal{P}(\Omega).$
\end{proof}

We end this section with two results showing that GOWA robust optimum value lies between the min-min robust optimum and min-max robust optimum values.

\begin{theorem}\label{boundbymin-max}
For any given $w_i^,$s and $\lambda,$\\
\begin{enumerate}
    \item [\emph{(a)}] GOWA robust optimal value of $\mathcal{P}(\Omega)$ never exceeds min-max robust optimal value.
    \item[\emph{(b)}] GOWA robust optimum value is greater than or equal to a min-min robust optimum value. 
\end{enumerate}
\end{theorem}

\begin{proof}
\begin{enumerate}
    \item [(a)] Let us suppose that $x^\ast$ and $ x_0$ be the GOWA robust optimal solution and min-max robust optimal solution, respectively, of $\mathcal{P}(\Omega)$.
Notice that for any $x_0 \in \mathcal{X},$
\begin{align*}
    \phi(x_0) & = \left(\sum_{i=1}^{p}w_is_i(x_0)^{\lambda}\right)^{\frac{1}{\lambda}}
     \leq \left(\sum_{i=1}^{p}w_is_1(x_0)^{\lambda}\right)^{\frac{1}{\lambda}}
    = \left(s_1(x_0)^{\lambda}(\sum_{i=1}^{p} w_i)\right)^{\frac{1}{\lambda}}=s_1(x_0).
\end{align*}
As $x^\ast$ is GOWA robust optimal solution, $\phi(x^\ast) \leq \phi(x_0).$ So $\phi(x^\ast) \leq \phi(x_0) \leq s_1(x_0).$ As $s_1(x_0)$ is the min-max robust optimal value of $\mathcal{P}(\Omega),$ the result follows.
\item[(b)] Let us suppose that $x^\ast$ and $ x_0$ be the GOWA robust optimum solution and min-min robust optimum solution to $\mathcal{P}(\Omega),$ respectively.
Note that for any chosen $w_i^,$s and $\lambda,$
 \begin{align}\label{GOWA_lower_bound_with_minmin}
     \phi(x^\ast) &= \left(\sum_{i=1}^{p}w_is_i(x^\ast)^{\lambda}\right)^{\frac{1}{\lambda}}
    \geq \left(\sum_{i=1}^{p}w_is_p(x^\ast)^{\lambda}\right)^{\frac{1}{\lambda}}
     = \left(s_p(x^\ast)^{\lambda}(\sum_{i=1}^{p} w_i)\right)^{\frac{1}{\lambda}}= s_p(x^\ast)
 \end{align}
 As $x_0 \in \mathcal{X}$ is a min-min robust optimal solution, $s_p(x_0) \leq s_p(x^\ast).$ So, $s_p(x_0) \leq s_p(x^\ast) \leq \phi(x^\ast)$, i.e., $s_p(x_0) \leq \phi(x^\ast).$ As $s_p(x_0) $ is the min-min robust optimum value of $\mathcal{P}(\Omega),$ the result follows.
\end{enumerate}
\end{proof}

\section{Clarke subdifferential for GOWA robustness} \label{subdiffrential}
We first illustrate an example to show that the objective function $\phi,$ defined by \eqref{phi} may not be differentiable. Also, from Example \ref{GOWAbestexample}, we see that $\phi$ may not be a convex function. Additionally, at the point where $\phi$ is not differentiable, we develop the concept of Clarke subdifferentials \cite{Clarke}.

\begin{example}\label{GOWA_nondifferentiable} 
Let the uncertainty set be $\Omega = \{\zeta_1,\zeta_2\}$ and the feasible set be $\mathcal{X}=[4,6].$ Define $f:\mathcal{X}\times \Omega \to \mathbb{R}$ as
$$f(x,\zeta_1)=e^{x-5}~~~\text{ and }~~~f(x,\zeta_2)=e^{-(x-5)},$$
and take $w_1=w_2=0.5,~\lambda = 1.$ Then,
\begin{equation*}
\phi(x) =\left\{
        \begin{array}{ll}
            \frac{1}{2}e^{-(x-5)}, & \quad ~~~~\text{ if }~~~ x \in [4, 5]\\ 
            \frac{1}{2}e^{x-5}, & \quad ~~~~\text{ if }~~~  x \in [5, 6].
        \end{array}
    \right.
\end{equation*}
Clearly, $\phi(x)$ is not differentiable at $x=5.$
\end{example}

For a given $\bar{x},$ it may be possible that there exists $k_1, k_2, \ldots, k_q \in \{1,2,\ldots,p\}$ such that $f(\bar{x}, \zeta_{k_1})=f(\bar{x},\zeta_{k_2})=\cdots=f(\bar{x},\zeta_{k_q}).$ The following lemma gives the maximum function in the neighborhood of $\bar{x}$, which is helpful in determining the subdifferential of the objective function $\phi$ at $\bar{x}$.
\begin{lemma}\label{lemma_1}
     Consider $\mathcal{P}(\Omega)$ with each $f(x,\zeta_i)$ is continuous on $\mathcal{X},~i=1,2,\ldots,p.$
      Let $\ell$ be the least index in \{1,2,\ldots,p\} for which
     \[
     s_\ell(\bar{x})= f(\bar{x}, \zeta_{k_1})=f(\bar{x},\zeta_{k_2})=\cdots=f(\bar{x},\zeta_{k_q}).
     \]
Then, there exists a $\delta >0$ such that 
\begin{equation}\label{sl_x_in_delta}
s_\ell(x)=\max~\{f(x,\zeta_{k_1}),f(x, \zeta_{k_2}),\ldots,f(x,\zeta_{k_q})\}, \text{ for all } x \in N_{\delta}(\bar{x}),
\end{equation}
where $N_{\delta}(\bar{x})= \{x : \|\bar{x}-x\| < \delta\}$
\end{lemma}

\noindent

\begin{proof}
    On the contrary, suppose that there does not exist any  $\delta>0$ such that \eqref{sl_x_in_delta} holds. Therefore, corresponding to any chosen $\delta>0$, there exists an  $x_\delta$ in $N_\delta(\bar x)$ such that $s_\ell(x_\delta)=f(x_\delta,\zeta_{j(x_\delta)}) \neq \max\{f(x_\delta,k_1),f(x_\delta,k_2),$ $\ldots,f(x_\delta,k_q)\}$. Now, we consider a sequence of neighbourhoods at $\bar{x}:$  $N_{\delta_n}(\bar{x})=\{x : \|\bar{x}-x\| < \delta_n\}$ with $\delta_n=\tfrac{1}{n}$, $n\in\mathbb{N}$. Then, corresponding to each $\delta_n,$ we get $x_n \in N_{\delta_n}(\bar{x})$ such that
    \[
    s_\ell(x_n)=f(x_n,\zeta_{j(x_n)}),~~j(x_n) \in \{\{1,2,\ldots,p\}\setminus\{k_1,k_2,\ldots,k_q\}\},~\text{ for all } n\in\mathbb{N}.
    \]
    As $n \to \infty,
    s_\ell(x_n) \to s_\ell(\bar{x})$ and $f(x_n, \zeta_{j(x_n)}) \to f(\bar{x}, \zeta_{j(\bar{x})})$ for some \\ $j(\bar{x}) \in \{\{1,2,\ldots,p\}\setminus\{k_1,k_2,\ldots,k_q\}\}, $
    we get $s_\ell(\bar{x})=f(\bar{x},\zeta_{j(\bar{x})})$ for some $j(\bar{x}) \in \{\{1,2,\ldots,p\}\setminus\{k_1,k_2,\ldots,k_q\}\}.$
    This is a contradiction.
\end{proof}

\begin{remark}\label{repsentation_of_s_i+m}
   In Lemma \ref{lemma_1}, let us consider an index set $\mathcal{J}_{\ell}(x)=\{j \in \{1,2,\ldots,q\}: f(x,\zeta_{k_j})=s_{\ell+t}(x),~ 0 \leq t \leq q-\ell,\text{ for all } x \in N_\delta(\bar{x})\}.$ If $f(x,\zeta_{k_{j(x)}})=s_{\ell+m-1}(x) \text{ for all } x \in N_{\delta}(\bar{x}),$ then the function $s_{\ell+m},~ 0\leq m \leq q-\ell,$ can be represented by
   \[
   s_{\ell+m}(x)= \max~\{f(x,\zeta_{k_j}) : j \in \mathcal{J}_{\ell+m}'(x)\},
   \]
   where $\mathcal{J}'_{\ell+m}(x)=\mathcal{J}_{\ell}(x)\setminus\{j(x)\}~~\text{for all}~~x \in N_{\delta}(\bar{x}).$
\end{remark}

\noindent
In the following example, we exemplify Lemma \ref{lemma_1}.
\begin{example}\label{example_to_find_subdifferential}
Let the uncertainty set be $\Omega=\{\zeta_1, \zeta_2, \zeta_3, \zeta_4, \zeta_5\}$ and the feasible set be $\mathcal{X}=[2,6]$.
We define $f:\mathcal{X}\times\Omega \to \mathbb{R}$
by $f(x,\zeta_1)=x, f(x,\zeta_2)=|\tfrac{x-4}{3}|+4, f(x,\zeta_3)=-|\tfrac{x-4}{2}|+4, f(x,\zeta_4)=(x-4)^2+1$ and $f(x, \zeta_5)=-\tfrac{x-4}{3}+5.$ Graphs of $f(x,\zeta_1), f(x,\zeta_2), f(x,\zeta_3), f(x,\zeta_4)$ and $f(x,\zeta_5)$ are drawn in Figure \ref{example_to_find_sub_a}. Here
\allowdisplaybreaks
\begin{align*}
    s_1(x)
    & = \begin{dcases}
        -\tfrac{x-4}{3}+5, & ~\text{if}~~ x \in [2,4.750] \\
        x, & ~\text{if}~~ x \in  [4.750,6],
    \end{dcases}\\
    s_2(x) &= \begin{dcases}
        (x-4)^2+1, & ~\text{if}~~ x \in [2,2.093]\cup [5.907,6]\\
        |\tfrac{x-4}{3}|+4, & ~\text{if}~~ x \in  [2.093,4]\cup [5.500,5.907]\\
        x, & ~\text{if}~~ x \in  [4,4.75]\\
        -\tfrac{x-4}{3}+5, & ~\text{if}~~ x \in  [4.75,5.5],
    \end{dcases}\\
        s_3(x)
    & = \begin{dcases}
        |\tfrac{x-4}{3}|+4, & ~\text{if}~~ x \in [2,2.093]\cup[4,5.5]\cup[5.907,6]\\
        (x-4)^2+1, & ~\text{if}~~ x \in  [2.093,2.5]\cup[5.84,5.907]\\
        -|\tfrac{x-4}{2}|+4, & ~\text{if}~~ x \in  [2.5,4]\\
        -\tfrac{x-4}{3}+5, & ~\text{if}~~ x \in  [5.5,5.84]\cup[-0.459,0.910]\cup [1.405,1.781],
    \end{dcases}\\
     s_4(x)&= \begin{dcases}
        -|\tfrac{x-4}{2}|+4, & ~\text{if}~~ x \in [2,2.5]\cup[4,5.5]\\
        (x-4)^2+1, & ~\text{if}~~ x \in  [2.5,2.697]\cup[5.5,5.84]\\
        x, & ~\text{if}~~ x \in  [2.697,4]\\
        -\tfrac{x-4}{3}+5, & ~\text{if}~~ x \in  [5.84,6]
    \end{dcases}\\
    \text{and}~~~s_5(x)&= \begin{dcases}
        x, & ~\text{if}~~ x \in [2,2.697]\\
        (x-4)^2+1, & ~\text{if}~~ x \in  [2.697,5.5]\\
        -|\tfrac{x-4}{2}|+4, & ~\text{if}~~ x \in  [5.5,6].
    \end{dcases}
\end{align*}
Graphs of $s_1(x),s_2(x),s_3(x),s_4(x)$ and $s_5(x)$ are drawn in Figure \ref{example_to_find_sub_b}.
At $\bar{x}=4,$ we have $f(4,\zeta_1)=f(4,\zeta_2)=f(4,\zeta_3).$ Here $k_1=1, k_2=2,\text{ and } k_3=3.$ Also, $s_2(4)=s_3(4)=s_4(4)=f(x,\zeta_1)=f(x,\zeta_2)=f(x,\zeta_3).$ Therefore, $\ell= \min \{2,3,4\}=2.$ We choose $\delta = \tfrac{1}{2}.$ Then, from Lemma \ref{lemma_1}, we have
\[
s_2(x)=\max~\{f(x,\zeta_1),f(x,\zeta_2),f(x,\zeta_3),~ \text{ for all } x \in [3.5, 4.5]\}.
\]
Notice that $\mathcal{J}_2(x)=\{1,2,3\}$ for all $x \in [3.5, 4.5].$ For $m=1,$ we have
\[\mathcal{J}'_3(x)=\begin{dcases}
        \{1,3\}, & ~\text{if}~~ x \in [3.5,4]\\
        \{2,3\}, & ~\text{if}~~ x \in  [4,4.5].
        \end{dcases}\]
Therefore, from Remark \ref{repsentation_of_s_i+m}, we have
 \begin{align*}
        s_3(x)&= \max\{f(x,\zeta_{k_j}) : j \in \mathcal{J}'_{3}(x)\},\\
        &= \max \left\{f(x, \zeta_{k_j}): j \in \left.\begin{dcases}
        \{1,3\}, & ~\text{if}~~ x \in [3.5,4]\\
        \{2,3\}, & ~\text{if}~~ x \in  [4,4.5],
    \end{dcases}\right\}\right\}\\
       &= \begin{dcases}
        \max\begin{dcases}
         -|\tfrac{x-4}{2}|+4, &~\text{if}~ x \in [3.5,4] \\
        x, &~\text{if}~ x \in  [3.5,4]
    \end{dcases}\\
        \max\begin{dcases}
        -|\tfrac{x-4}{2}|+4, &~\text{if}~ x \in [4,4.5] \\
        |\tfrac{x-4}{3}|+4, &~\text{if}~ x \in  [4,4.5] 
    \end{dcases}
    \end{dcases}\\
    &=\begin{dcases}
        -|\tfrac{x-4}{2}|+4, &~\text{if}~ x \in [3.5,4] \\
        |\tfrac{x-4}{3}|+4, &~\text{if}~ x \in  [4,4.5]. 
    \end{dcases}
 \end{align*}
 For $m=2,$ we have $\mathcal{J}'_4(x)=\begin{cases}
        \{1\}, &~\text{if}~ x \in [3.5,4] \\
        \{3\}, &~\text{if}~ x \in  [4,4.5].
    \end{cases}$ 
Therefore,
$ s_4(x)=\begin{cases}
        x, &~\text{if}~ x \in [3.5,4] \\
        -|\tfrac{x-4}{2}|+4, &~\text{if}~ x \in  [4,4.5]. 
    \end{cases}
$ 
\begin{figure}[H]
    \centering
    \subfloat[The objective functions for different scenarios\label{example_to_find_sub_a}]{\includegraphics[scale=0.4]{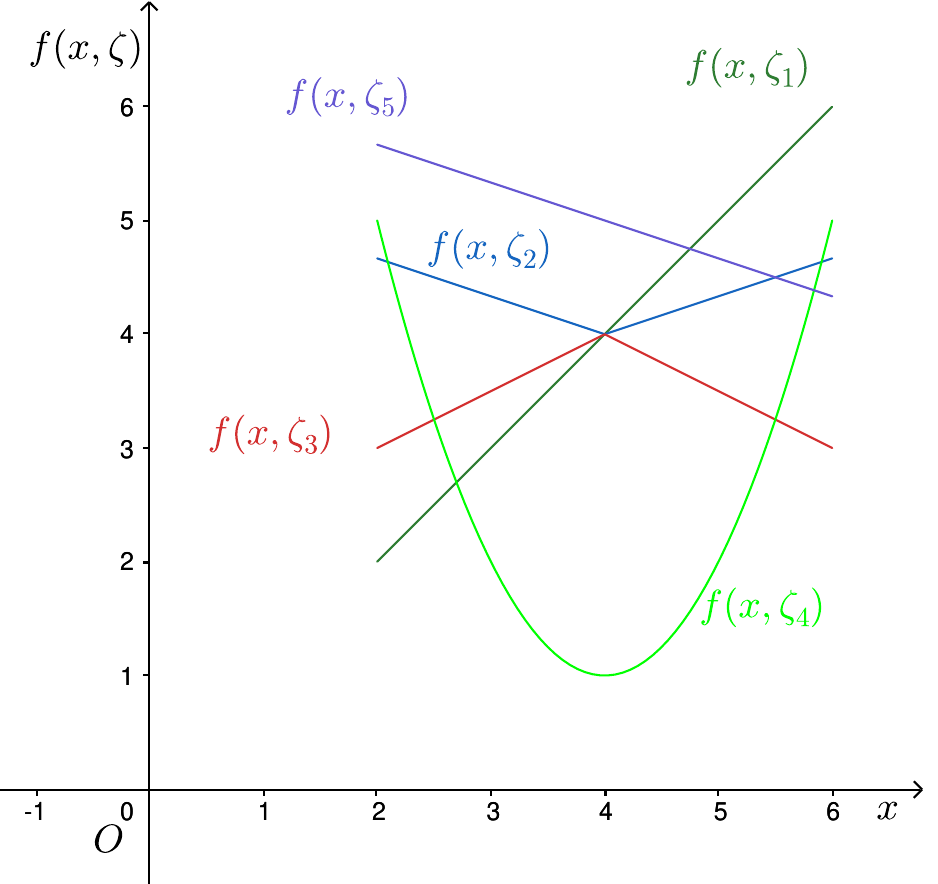}}
    \qquad \qquad 
    \subfloat[The function $s_1, s_2, s_3, s_4$ and $s_5$\label{example_to_find_sub_b}]{\includegraphics[scale=0.4]{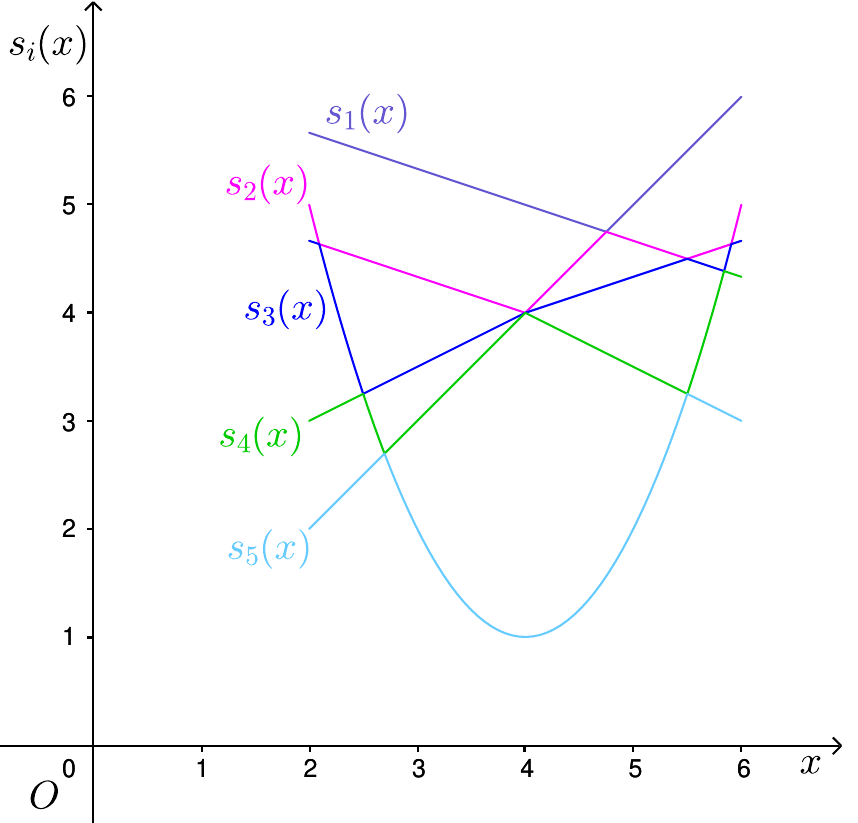}}
    \caption{Functions involved in GOWA robust counterpart for Example \ref{example_to_find_subdifferential}}
\end{figure}
\end{example}

\begin{theorem}\emph{\cite{Bagirov}}\label{sub_of_max_function} Let $g_j : \mathbb{R}^n \to \mathbb{R}$ be locally Lipschitz continuous functions at $\bar{x}\in\mathbb{R}^{n}$ for all $j=1,2,\ldots,m$. Then, the function
$$g(x)= \max~\{g_j(x) : j=1,2,\ldots,m\}$$
is locally Lipschitz continuous at $\bar{x}$ and 
\begin{align}\label{max}
\partial g(\bar{x}) \subseteq \text{conv}\{\partial g_j(\bar{x}) : j \in \mathcal{J}(\bar{x})\},
\end{align}
where $\mathcal{J}(\bar{x})=\{j \in \{1,2,\ldots,m\} : g_j(\bar{x})=g(\bar{x})\}.$
Moreover, if $g_j$ are subdifferentially regular at $\bar{x}$ for all $j= 1,2,\ldots,m,$ then $g$ is also subdifferentially regular at $\bar{x}$ and equality holds in \emph{(\ref{max})}.
\end{theorem}
\begin{theorem}\label{sub_s_i}
For a given uncertain single objective optimization problem $\mathcal{P}(\Omega),$ suppose for a given $\bar{x}$ there exists $k_1, k_2, \ldots, k_q \in \{1,2,\ldots,p\}$ such that $f(\bar{x}, \zeta_{k_1})=f(\bar{x},\zeta_{k_2})=\cdots=f(\bar{x},\zeta_{k_q}).$ Let $\ell$ be the least index in $\{1,2,\ldots,p\}$ for which
     \[
     s_\ell(\bar{x})= f(\bar{x}, \zeta_{k_1})=f(\bar{x},\zeta_{k_2})=\cdots=f(\bar{x},\zeta_{k_q}).
     \] 
Then, the subdifferential of the function $s_\ell(x)$ at $\bar{x}$ and 
\begin{align}\label{s_i(x)sub}
    \partial s_\ell(\bar{x}) \subseteq \text{conv}\{\partial f(\bar{x},\zeta_{k_j}) : j \in \mathcal{J}_\ell(\bar{x})\},
\end{align}
where $\mathcal{J}_\ell(\bar{x})= \{j \in \{1,2,\ldots,q\} : f(\bar{x},\zeta_{k_j})=s_\ell(\bar{x})\}.$
Moreover, if each $f(x,\zeta_{k_j}), j=1,2,\ldots,q$ is subdifferentially regular at $\bar{x},$ then $s_\ell(x)$ is also subdifferentially regular at $\bar{x}$ and equality holds in \emph{(\ref{s_i(x)sub})}.
\end{theorem}
%%%%%%%%%%%%%%%%%%%%%%%%%%%%%%%%%%%%%%%%%%%%%%%%%%%%%%%%%%
\begin{proof}
To find the subdiffrential of the function $s_l(x)$ at $\bar{x},$ from Lemma \ref{lemma_1}, there exists a $\delta$-neighbourhood $N_{\delta}(\bar{x})$ at $\bar{x}$ such that the function $s_l(x)$ becomes
 \[
 s_\ell(x)=\max~\{f(x,\zeta_{k_1}),f(x,\zeta_{k_2}),\ldots,f(x,\zeta_{k_q})\} \text{ for all } x \in N_{\delta}(\bar{x}),
 \]
where $N_{\delta}(\bar{x}) =\{x : \|\bar{x}-x\| < \delta\}.$\\ 
Consider the functions $g_\ell : \mathbb{R}^q \to \mathbb{R}$ and $h : \mathcal{X} \to \mathbb{R}^q$ by 
\begin{align*}
    &g_\ell(y)=\underset{l=1,2,\dots,q}{\max}\{y_l\}~~\text{ and }~~
    h(x)=(f(x,\zeta_{k_1}), f(x,\zeta_{k_2}),\dots,f(x,\zeta_{k_q})).
\end{align*}
Now we have $s_\ell(x)=(g_\ell \circ h)(x)$ in $N_{\delta}(\bar{x}).$ For all $y,y' \in \mathbb{R}^q$ and $\mu \in [0,1],$
\begin{align*}
    g_\ell(\mu y+(1-\mu)y')&= \underset{l=1,2,\ldots,q}{\max}\{\mu y_l +(1-\mu)y'_j\}\\
    & \leq \mu \underset{l=1,2,\ldots,q}{\max}\{y_l\}+(1-\mu) \underset{l=1,2,\ldots,q}{\max}\{y'_l\}\\
    & = \mu g_\ell(y)+(1-\mu) g_\ell(y'),
\end{align*}

\noindent
which means that $g_\ell(y)$ is convex and by Theorem $2.22$ (i) in {\cite{Bagirov}} locally Lipschitz continuous at $h(\bar{x}).$

\noindent
Let $\mathcal{L}_\ell(y)=\{l \in \{1,2,\ldots,q\} : y_l=g_\ell(y)\}.$
Now, the directional derivative of $g_\ell(y)$ along the direction $d \in \mathbb{R}^q$ is

\noindent
\begin{align*}
    g'_\ell(y;d) &= \underset{t \to 0}{\lim}~\frac{g_\ell(y+td)-g_\ell(y)}{t}\\
    &= \underset{t \to 0}{\lim}~\underset{l=1,2,\ldots,q}{\max}~ \frac{\{y_l+td_l\}-g_\ell(y)}{t}\\
    &= \underset{t \to 0}{\lim}~\underset{l \in \mathcal{L}_\ell(y)}{\max}~\frac{\{y_l+td_l\}-g_\ell(y)}{t}\\
    &= \underset{t \to 0}{\lim}~\underset{l \in \mathcal{L}_\ell(y)}{\max}~\frac{\{y_l+td_l-y_l\}}{t}.
\end{align*}
Thus, $g'_\ell(y;d)= \underset{l \in \mathcal{L}_\ell(y)}{\max} d_l$ and by Theorem $3.8$ (i) in  {\cite{Bagirov}}, we have $g^{\circ}_\ell=g'_\ell,$ which gives
$$\partial g_\ell(y)=\{\beta \in \mathbb{R}^q : \underset{l \in \mathcal{L}_\ell(y)}{\max}~ d_l \geq \beta^{\top}d \text{ for all } d \in \mathbb{R}^q\}.$$
This implies that any $\beta = (\beta_1,\beta_2,\ldots,\beta_q)$ in $\partial g_\ell (y)$ satisfies $\beta_l \geq 0, l=1,2,\ldots,q, \beta_l=0, \text{ if } l \notin \mathcal{L}_{\ell} (y),$ and $\sum_{l=1}^q \beta_{l}=1,$ so we can calculate the Clarke subdifferential of $g_\ell(y)$ at $h(\bar{x}) \in \mathbb{R}^q$ by 
\begin{align*}
\partial g_\ell(h(\bar{x}))&= \left\{\beta \in \mathbb{R}^q : \beta_l \geq 0, ~~\sum_{l=1}^{q}\beta_l=1~~\text{and}~~\beta_l=0 ~\text{if}~l \notin \mathcal{L}_\ell(h(\bar{x}))\right\}\\
\iff \partial g_\ell(h(\bar{x}))&= \left\{\beta \in \mathbb{R}^q : \beta_j \geq 0, ~~\sum_{j=1}^{q}\beta_j=1~~\text{and}~~\beta_j=0 ~\text{if}~j \notin \mathcal{J}_\ell(\bar{x})\right\}.
\end{align*}
By Theorem $3.19$ in {\cite{Bagirov}} on $s_\ell(\bar{x}),$ we get
\begin{align*}
    \partial s_\ell(\bar{x}) &\subseteq \text{conv} \left\{\sum_{j=1}^{q}\beta_j \zeta_j : \zeta_j \in \partial h_j(\bar{x})~~\text{and}~~\beta \in \partial g_\ell(h(\bar{x}))\right\}\\
   & =\text{conv}\left\{\underset{j \in \mathcal{J}_\ell(\bar{x})}{\sum}\beta_j \partial f(\bar{x},\zeta_{k_j}): \beta_j \geq 0 ~~\text{and}~~\underset{j \in \mathcal{J}_\ell(\bar{x})}{\sum}\beta_j =1\right\}\\
   & =\text{conv}\left\{\partial f(\bar{x}, \zeta_{k_j}): j \in \mathcal{J}_\ell(\bar{x})\right\}.
\end{align*}

\noindent
Moreover, we suppose that each $f(x,\zeta_{k_j}),  j=1,2,\ldots,q$ is subdifferentially regular at $\bar{x}.$ As $g_\ell(y)$ is convex, so by Theorem  $3.13$ (ii) in {\cite{Bagirov}}, it is subdifferentially regular at $h(\bar{x}).$ Hence, Theorem $3.20$ (i) in \cite{Bagirov} implies that $s_i(x)$ are subdifferentially regular at $\bar{x}$ and equality holds in \eqref{s_i(x)sub}.
\end{proof}

Let the function $s_i$ be as given in {(\ref{s_1})} and {(\ref{s_i})}. Suppose the value of $s_i$ at $\bar{x}$ is
    $
     s_i(\bar{x})= f(\bar{x}, \zeta_{k_1})=f(\bar{x},\zeta_{k_2})=\cdots=f(\bar{x},\zeta_{k_q}), \text{ for least } i =1,2,\ldots,p.
     $
In the next theorem, we find the subdifferential of $s_{i+m}$ at $\bar{x},~i+m=1,2,\ldots,q,$ before that we give a remark to difine the subdifferential function $s_{i+m}$ in a $\delta-$neighbourhood.

\noindent
\begin{theorem}\label{sub_i+m}
      For a given uncertain single objective optimization problem $\mathcal{P}(\Omega),$ suppose for a given $\bar{x}$ there exist $k_1, k_2, \ldots, k_q \in \{1,2,\ldots,p\}$ such that $f(\bar{x}, \zeta_{k_1})=f(\bar{x},\zeta_{k_2})=\cdots=f(\bar{x},\zeta_{k_q}).$ Let $\ell$ be the least index in $\{1,2,\ldots,p\}$ for which
     \[
     s_\ell(\bar{x})= f(\bar{x}, \zeta_{k_1})=f(\bar{x},\zeta_{k_2})=\cdots=f(\bar{x},\zeta_{k_q}).
     \] 
Then, the subdifferential of the function $s_{\ell+m}$ at $\bar{x},$  is
\begin{align}\label{s_i+1sub}
\partial s_{\ell+m}(\bar{x}) \subseteq \text{conv}~ \{ \partial f(\bar{x}, \zeta_{k_i}) : i \in \mathcal{J}'_{\ell+m}(\bar{x})\},~~~0 \leq m \leq q-\ell,
\end{align}
where $\mathcal{J}'_{\ell+m}(x)=\mathcal{J}_{\ell}(x)\setminus\{j(x)\}$ and the index set $\mathcal{J}_{\ell}(x)=\{j \in \{1,2,\ldots,q\}: f(x,\zeta_{k_j})=s_{\ell+t}(x),~ 0 \leq t \leq q-\ell\}.$  
Moreover, if each $f(x,\zeta_{k_j}), j=1,2,\ldots,q$ is subdifferentially regular at $\bar{x},$ then $s_{\ell+m}(x)$ is also subdifferentially regular at $\bar{x}$ and equality holds in \emph{(\ref{s_i+1sub})}.
\end{theorem}
\begin{proof}
To find the subdifferential of the function $s_{\ell+m},$ from Lemma \ref{lemma_1}, there exists a\\ $\delta$-neighbourhood $N_{\delta}(\bar{x})$ at $\bar{x}$ such that (\ref{sl_x_in_delta}) holds. Let us consider an index set $\mathcal{J}_{\ell}(x)=\{j \in \{1,2,\ldots,q\}: f(x,\zeta_{k_j})=s_{\ell+t}(x),~ 0 \leq t \leq q-\ell,\text{ for all } x \in N_\delta(\bar{x})\}.$ If $f(x,\zeta_{k_{j(x)}})=s_{\ell+m-1}(x) \text{ for all }\\ x \in N_{\delta}(\bar{x}),$ then the function $s_{\ell+m},~ 0 \leq m \leq q-\ell$ can be represented by \[s_{\ell+m}(x)= \max~\{f(x,\zeta_{k_j}) : j \in \mathcal{J}'_{\ell+m}(x) \text{ for all } x \in N_{\delta}(\bar{x})\},\]
   where $\mathcal{J}'_{\ell+m}(x)=\mathcal{J}_{\ell}(x)\setminus\{j(x)\}.$\\
Now, consider the functions $g_{\ell+m} : \mathbb{R}^q \to \mathbb{R}$ and $h : \mathcal{X} \to \mathbb{R}^q$ given by 
\begin{align*}
    &g_{\ell+m}(y)= \underset{l=1,2,\ldots,q}{\max}\{\{y_l\}\setminus\{g_{\ell+m-1}(y)\}\}~~\text{ and }~~
    h(x)=(f(x,\zeta_{k_1}), f(x,\zeta_{k_2}),\dots,f(x,\zeta_{k_q})).
\end{align*}
then $s_{\ell+m}(x)=(g_{\ell+m} \circ h)(x)$ in $N_{\delta}(\bar{x}).$ Note that for all $y,y' \in \mathbb{R}^q$ and $\mu \in [0,1],$ we have
\begin{align*}
    g_{\ell+m}(\mu y +(1-\mu)y') &= \underset{l=1,2,\ldots,q}{\max}\{\{\mu y_l +(1-\mu)y'_l\}\setminus\{g_{\ell+m-1}(\mu y +(1-\mu y')\}\}\\
    & \leq \mu \underset{l=1,2,\ldots,q}{\max}\{\{y_l\}\setminus\{g_{\ell+m-1}(y)\}+(1-\mu)\underset{l=1,2,\ldots,q}{\max}\{\{y'_l\}\setminus\{g_{\ell+m-1}(y')\}\\
    &=\mu g_{\ell+m}(y)+(1-\mu)g_{\ell+m}(y'),
\end{align*}
which means that $g_{\ell+m}(y)$ is convex and by Theorem $2.22$ (i) in {\cite{Bagirov}}, it is locally Lipschitz continuous at $h(\bar{x}).$ Let $\mathcal{L}_{\ell+m}(y)=\{l \in \{1,2,\ldots,q\}:y_l = g_{\ell+m}(y)\}.$
Then, the directional derivative of $g_{\ell+m}(y)$ is
\allowdisplaybreaks
\begin{align*}
    g'_{\ell+m}(y;d)&= \underset{t \to 0}{\lim}~ \frac{g_{\ell+m}(y+td)-g_{\ell+m}(y)}{t}\\
    & = \underset{t \to 0}{\lim} \underset{l=1,2,\ldots,q}{\max}\frac{\{\{y_l+td_l\}\setminus\{g_{\ell+m-1}(y)\}\}-g_{\ell+m}(y)}{t}\\
    &= \underset{t \to 0}{\lim} \underset{l \in \mathcal{L}_{\ell+m}(y)}{\max} \frac{y_l+td_l-y_l}{t}\\
    &= \underset{l \in \mathcal{L}_{\ell+m}(y)}{\max} ~d_l,
\end{align*}
and by Theorem $3.8$ (i) in {\cite{Bagirov}}, we have $g^{\circ}_{\ell+m}=g'_{\ell+m},$ which gives
$$\partial g_{\ell+m}(y)=\{\beta \in \mathbb{R}^q : \underset{l \in \mathcal{L}_{\ell+m}(y)}{\max}~ d_l \geq \beta^{\top}d~\text{ for all }~ d \in \mathbb{R}^q\}$$
This implies that any $\beta = (\beta_1,\beta_2,\ldots,\beta_q)$ in $\partial g_{\ell+m} (y)$ satisfies $\beta_l \geq 0, l=1,2,\ldots,q, \beta_l=0, \text{ if } l \notin \mathcal{L}_{\ell+m} (y),$ and $\sum_{l=1}^q \beta_{l}=1,$ and so we can calculate the clarke subdifferential of $g_{\ell+m}(y)$ at $h(\bar{x}) \in \mathbb{R}^q$ by 
\begin{align*}
\partial g_{\ell+m}(h(\bar{x}))&= \left\{\beta \in \mathbb{R}^q : \beta_l \geq 0, ~~\sum_{l=1}^{q}\beta_l=1~~\text{and}~~\beta=0 ~\text{if}~l \notin \mathcal{L}_{\ell+m}(\bar{x})\right\}\\
\iff \partial g_{\ell+m}(h(\bar{x}))&= \left\{\beta \in \mathbb{R}^q : \beta_j \geq 0, ~~\sum_{j=1}^{q}\beta_j=1~~\text{and}~~\beta=0 ~\text{if}~j \notin \mathcal{J}_{\ell+m}'(\bar{x})\right\}.
\end{align*}
Applying Theorem $3.19$ in {\cite{Bagirov}} on $s_{\ell+m}(x),$ we get
\begin{align*}
    \partial s_{\ell+m}(\bar{x}) &\subseteq \text{conv}~ \left\{\sum_{j=1}^{q}\beta_j \zeta_j : \zeta_j \in \partial h_j(\bar{x})~~\text{and}~~\beta \in \partial g_{\ell+m}(h(\bar{x}))\right\}\\
  & =\text{conv}~ \left\{\underset{j \in \mathcal{J}_{\ell+m}'(\bar{x})}{\sum}\beta_j \partial f(\bar{x},\zeta_{k_j}): \beta_j \geq 0 ~~\text{and}~~\underset{j \in \mathcal{J}_{\ell+m}'(\bar{x})}{\sum}\beta_j =1\right\}\\
  & =\text{conv} \left\{\partial f(\bar{x}, \zeta_{k_j}): j \in \mathcal{J}_{\ell+m}'(\bar{x})\right\}.
\end{align*}
Moreover, we suppose that each $f(x,\zeta_{k_j}),  j=1,2,\ldots,q,$ is subdifferentially regular at $\bar{x}.$ As $g_{i+m}(y)$ is convex, so by Theorem $3.13$ (ii) in {\cite{Bagirov}}, it is subdifferentially regular at $h(\bar{x}).$ Hence, Theorem $3.20$ (i) in \cite{Bagirov} implies that $s_{\ell+m}(x)$ are subdifferentially regular at $\bar{x}$ and equality holds in \eqref{s_i+1sub}.
\end{proof}
\begin{example}
Consider $f(x,\zeta_i), \text{ and } s_i,~ i=1,2,\ldots,5,$ are as given in  Example \ref{example_to_find_subdifferential}. We calculate Clarke subdifferential of each function $s_i, i=1,2,\ldots,5$ at $\bar{x}=4.$
Let $\delta=\tfrac{1}{2}.$ In the $\delta$-neighbourhood,

\begin{enumerate}
\item[(i)] For $i=1,$ $s_1(4)=f(4,\zeta_5).$ From Lemma \ref{lemma_1}, we have 
$s_1(x)=-\tfrac{x-4}{3}+5, \text{ for all } x \in [3.5,4.5].$ Clearly, $\partial s_1(4)=\{-\tfrac{1}{3}\}.$

\item[(ii)] For $i=2,$ we have 
\begin{align*}
    s_2(x)&=\max~\{f(x,\zeta_1),f(x,\zeta_2),f(x,\zeta_3)\}\\
         &=\max\left\{x, |\tfrac{x-4}{3}|+4, -|\tfrac{x-4}{2}|+4\right\}~~~\text{ for all } x \in [3.5,4.5]. 
\end{align*}
Therefore, from Theorem \ref{sub_s_i}, $\partial s_2(4)=\text{conv}\left\{1,[-\tfrac{1}{3},\tfrac{1}{3}],[\tfrac{-1}{2},\tfrac{1}{2}]\right\}=\left[\tfrac{-1}{2},1\right].$

\item[(iii)] For $i=3,$ we have 
\begin{align*}
    s_3(x)&=\begin{dcases}
        -|\tfrac{x-4}{2}|+4, &~\text{if}~ x \in [3.5,4] \\
        |\tfrac{x-4}{3}|+4, &~\text{if}~ x \in  [4,4.5]. 
    \end{dcases}
\end{align*}
Therefore, from Theorem \ref{sub_i+m}, $\partial s_3(4)=\text{conv}\left\{\tfrac{1}{2},\tfrac{1}{3}\right\}=\left[\tfrac{1}{3},\tfrac{1}{2}\right].$

\item[(iv)] For $i=4,$ we have 
\begin{align*}
    s_4(x)&=\begin{dcases}
        x, &~\text{if}~ x \in [3.5,4] \\
        -|\tfrac{x-4}{2}|+4, &~\text{if}~ x \in  [4,4.5]. 
    \end{dcases}
\end{align*}
Therefore, from Theorem \ref{sub_i+m}, $\partial s_4(4)=\text{conv}\left\{1,-\tfrac{1}{2}\right\}=\left[-\tfrac{1}{2},1\right].$

\item[(v)] for $i=5,$ $s_5(4)=f(4,\zeta_4).$ From Lemma \ref{lemma_1}, we have 
\begin{align*}
    s_5(x)&=(x-4)^2+1,~~~\text{if}~ x \in [3.5,4.5].
\end{align*}
Therefore, from Theorem \ref{sub_s_i}, $\partial s_5(4)=\{0\}.$
\end{enumerate}
\end{example}

\section{Conclusion and future scopes}\label{conclusion}
In this article, we have given a new concept of robustness for solving uncertain single objective optimization problems called GOWA robustness. Also, some important characteristics of the GOWA objective function, like continuity [Theorem \ref{continuity}], coerciveness [Theorem \ref{coeirciveness}], monotonicity  [Theorem \ref{monotonicity_of_phi}], and its behaviour as $\lambda \to 0+$ and $\lambda \to 1$ [Note \ref{lambdazero}] have been presented. Next, we have shown how GOWA robustness is related to other existing robustness like min-max [Proposition \ref{min-maxthm}], min-min [Proposition \ref{min-minthm}], and lightly robustness [Proposition \ref{light_robust_optimality}]. In Theorem \ref{boundbymin-max}, we have concluded that the optimal value of an uncertain single objective optimization problem by GOWA robustness always lies between min-min and min-max robustness. As it is found that GOWA objective function need not be differentiable [Example \ref{GOWA_nondifferentiable}], we have given an analysis to find the Clarke subdifferential set [Theorems \ref{sub_s_i} and \ref{sub_i+m}] for GOWA objective function.

Further research can be carried out to discuss suitable models in real-world applications. In future research, we would like to focus on the following directions:
\begin{itemize}
    \item Current study motivates us to analyze several conventional concepts of robustness in greater depth to find the solution to uncertain optimization problems. Various algorithms may serve as valuable tools for capturing GOWA robust optimal solutions.
    \item We will try to develop algorithms, like A. Upadhayay et al. developed cone method in \cite{Upadhayay2023}, to capture solutions to uncertain optimization problems.
    \item Moreover, future studies can explore the theoretical properties as well as the experimental performance of the GOWA robustness for uncertain single-objective optimization problems. 
    \item With the help of the Clarke subdifferential analysis of the GOWA robust objective function, one can endeavour to perform rigorous nonsmooth optimization of single-objective uncertain optimization problems. 
    \item One can introduce and analyze the concept of GOWA robust efficient solutions for multi-objective uncertain optimization problems. It will convert into a set optimization problem, then use methods like those given by K. Kumar et al. in \cite{K. Kumar}. 
\end{itemize}

\section*{Acknowledgement}
\noindent
Debdas Ghosh acknowledges the financial support of the research projects\\ MATRICS (MTR/2021/000696) and CRG (CRG/2022/001347) from the Science and Engineering Research Board, India. Nand Kishor acknowledges a research fellowship from the University Grants Commission (Government of India).


%%%%%%%%%%%%%%%%%%%%%%%%%%%%%%%%%%%%%%%%%%%%%%%%%%%%%%%%%%%%%%%%%%%%%%%%%%%%%%%%


\begin{thebibliography}{90}




\bibitem{Robustapplication}
 A. Ben-Tal and A. Nemirovski, \emph{Robust optimization–methodology and applications,} Math. Program. \textbf{92} (2002), 453--480.



\bibitem{Yu}
P. Kouvelis and G. Yu, \emph{Robust Discrete Optimization and its Applications}, Springer Science \& Business Media, \textbf{14} (2013).


\bibitem{Tiedemann}
J. Ide, M. Tiedemann, S. Westphal, and F. Haiduk, \emph{An application of deterministic and robust optimization in the wood cutting industry,} 4OR. \textbf{13} (2015), 35--57.

\bibitem{Eltoukhy}
A. E. Eltoukhy, Z. X. Wang, F. T. Chan, S. H. Chung, H. L. Ma, and X. P. Wang, \emph{Robust aircraft maintenance routing problem using a turn-around time reduction approach,}	IEEE Trans. Syst. Man Cybern.: Syst. \textbf{50(12)} (2019), 4919--4932.


\bibitem{Kouvelis}
P. Kouvelis and G. Yu, \emph{Easily Solvable Cases of Robust Discrete Optimization Problems,} Robust Discrete Optimization and Its Applications, (1997), 116--152.


\bibitem{Birge}
J. R. Birge and F. Louveaux, \emph{Introduction to Stochastic Programming}, Springer Science \& Business Media, (2011).




\bibitem{Ben}
A. Ben-Tal, L. El Ghaoui, and A. Nemirovski, \emph{Robust Optimization}, Princeton University Press, \textbf{28} (2009).




\bibitem{Fliege}
J. Fliege and R. Werner, \emph{Robust multiobjective optimization \& applications in portfolio optimization,} European J. Oper. Res. \textbf{234(2)} (2014), 422--433.





\bibitem{Kuhn}
K. Kuhn, A. Raith, M. Schmidt, and A. Schöbel, \emph{Bi-objective robust optimisation,} European J. Oper. Res. \textbf{252(2)} (2016), 418--431.





\bibitem{Robustness}
J. Ide and A. Schöbel, \emph{Robustness for uncertain multi-objective optimization: a survey and analysis of different concepts,} OR Spectrum. \textbf{38(1)} (2016), 235--271.


\bibitem{Irrigation}
M. S. Sabouni and M. Mardani, \emph{Application of robust optimization approach for agricultural water resource management under uncertainty,} J. Irrig. Drain. Eng. \textbf{139(7)} (2013), 571--581.



\bibitem{soyster}
A. S. Soyster, \emph{Convex programming with set-inclusive constraints and applications to inexact linear programming,} Oper. Res. \textbf{21(5)} (1973), 1154--1157.



\bibitem{Ben-Tal}
A. Ben-Tal and A. Nemirovski, \emph{Robust convex optimization,} Math. Oper. Res. \textbf{23(4)} (1998), 769--805.




\bibitem{Ben-Tal A}
A. Ben-Tal and A. Nemirovski, \emph{Robust solutions of uncertain linear programs,} Oper. Res. Lett. \textbf{25(1)} (1999), 1--13.

\bibitem{Ben-Tal B}
A. Ben-Tal and A. Nemirovski, \emph{Robust solutions of linear programming problems contaminated with uncertain data,} Math. Program. \textbf{88} (2000), 411--424.


\bibitem{Goerigk M}
M. Goerigk and A. Schöbel, \emph{Algorithm engineering in robust optimization,} in Algorithm Engineering: Selected Results and Surveys, (2016), 245--279.

\bibitem{Fischetti}%3
M. Fischetti and M. Monaci, \emph{Light Robustness,} in Robust and Online Large-scale Optimization, Springer Berlin, Heidelberg, (2009), 61--84.


\bibitem{Schobel}%4
A. Schöbel, \emph{Generalized light robustness and the trade-off between robustness and nominal quality,} Math. Methods Oper. Res. \textbf{80(2)} (2014), 161--191.


\bibitem{Benn}
A. Ben-Tal, A. Goryashko, E. Guslitzer, and A. Nemirovski, \emph{Adjustable robust solutions of uncertain linear programs,} Math. Program. \textbf{99(2)} (2004), 351--376.



\bibitem{Bertsmimas}
A. Ben-Tal, D. Bertsimas, and D. B. Brown, \emph{A soft robust model for optimization under ambiguity,} Oper. Res. \textbf{58(4-part-2)} (2010), 1220--1234.



\bibitem{Liebchen}
C. Liebchen, M. Lübbecke, R. Möhring, and S. Stiller, \emph{The concept of recoverable robustness, linear programming recovery, and railway applications,} Robust and Online Large-scale Optimization: Models and Techniques for Transportation Systems, Springer, Berlin, Heidelberg, (2009), 1--27.

\bibitem{Erera}
A. L. Erera, J. C. Morales, and M. Savelsbergh, \emph{Robust optimization for empty repositioning problems,} Oper. Res. \textbf{57(2)} (2009), 468--483.

\bibitem{Goerigk}
M. Goerigk and A. Schöbel, \emph{Recovery-to-optimality: A new two-stage approach to robustness with an application to aperiodic timetabling,} Comput. Oper. Res. \textbf{52} (2014), 1--15.


\bibitem{Yager}
R. R. Yager, \emph{Generalized OWA aggregation operators,} Fuzzy Optim. Decis. Mak. \textbf{3} (2004), 93--107.

\bibitem{Yager_2}
R. R. Yager, \emph{On the cardinality index and attitudinal character of fuzzy measures,} Int. J. Gen. Syst. \textbf{31(3)} (2002), 303--329.

\bibitem{Yager_1}
R. R. Yager, \emph{On ordered weighted averaging aggregation operators in multicriteria decision making,} IEEE Trans. Systems Man Cybernet. \textbf{18(1)} (1988), 183--190.


\bibitem{min-max}
M. Ehrgott, J. Ide, and A. Schöbel, \emph{Minmax robustness for multi-objective optimization problems,} European J. Oper. Res. \textbf{239(1)} (2014), 17--31.

\bibitem{Clarke}
F. H. Clarke, \emph{Optimization and Nonsmooth Analysis}, SIAM J. Control Optim. (1990).

\bibitem{ide}%4
J. Ide, E. Köbis, D. Kuroiwa, A. Schöbel, and C. Tammer, \emph{The relationship between multi-objective robustness concepts and set-valued optimization,} Fixed Point Theory Appl. \textbf{2014(1)} (2014), 1--20.


\bibitem{Bagirov}
A. Bagirov, N. Karmitsa, and M. M. Mäkelä, \emph{Introduction to Nonsmooth Optimization: Theory, Practice and Software}, Cham, Heidelberg: Springer International Publishing \textbf{12} (2014).



\bibitem{Upadhayay2023}
A. Upadhayay, D. Ghosh, Jauny, and N. Kishor, \emph{Cone Method for Uncertain Multiobjective Optimization Problems with Minmax Robustness}, in Combinatorial Optimization Under Uncertainty, CRC Press, (2023), 141--152.




\bibitem{K. Kumar}
K. Kumar, D. Ghosh, J. C. Yao, X. Zhao \emph{Nonlinear conjugate gradient methods for unconstrained set optimization problems whose objective functions have finite cardinality}, Optimization, (accepted) (2024).


\vspace{5cm}


















 










%%%%%%%%%%%%%%%%%%%%%%%%%%%%%%%%%%%%%%%%%%%%%%%%%%%%%











\end{thebibliography}
\end{document}